\documentclass[10pt]{article}
\usepackage[ansinew]{inputenc}
\usepackage{amsmath}
\usepackage{amsfonts}
\usepackage{amssymb}
\usepackage{amsthm}
\usepackage{newlfont}
\usepackage{amsbsy}
\usepackage{bm}
\usepackage{color}
\usepackage{epsfig}
\usepackage{dsfont}
\usepackage{latexsym}
\usepackage{array}
\usepackage{longtable}
\usepackage{enumerate}
\usepackage{multicol}
\usepackage{mathrsfs}
\usepackage{wasysym}
\usepackage{graphics}
\usepackage{graphicx}
\DeclareMathOperator\supp{supp}
\usepackage{hyperref}
\hypersetup{
    colorlinks,
    citecolor=blue,
    filecolor=black,
    linkcolor=red,
    urlcolor=black
}

\numberwithin{equation}{section}

\newtheorem{lemma}{Lemma}[section]
\newtheorem{theorem}{Theorem}[section]
\newtheorem{corol}{Corollary}[section]

\newtheorem{prop}{Proposition}[section]
\newtheorem{rem}{\it Remark}[section]

\title{\textbf{Well-posedness for Fractional Growth-Dissipative Benjamin-Ono Equations}}
\author{ Ricardo Pastr\'an \thanks{Universidad Nacional de Colombia, sede Bogot\'a, E-mail: {\tt rapastranr@unal.edu.co}} \\
Oscar G. Ria\~no C. \thanks{IMPA - Instituto de Matem\'atica Pura e Aplicada,
E-mail: {\tt ogrianoc@impa.br}}}

\begin{document}

\maketitle 

\begin{abstract}
This paper is devoted to study the Cauchy problem for the fractional growth-dissipative BO equations $u_t+\mathcal{H}u_{xx}-(D_x^{\alpha}-D_x^{\beta})u+uu_x=0$. For a wide class of parameter $\beta>1$ and $0<\alpha <\beta$, taking into account dispersive and dissipative effects, we establish sharp well-posedness results in Sobolev spaces $H^s(\mathbb{R})$ and $H^s(\mathbb{T})$ which yield new well-posedness conclusions for some physical relevant equations. In addition, we study the behavior of solutions as $\alpha \to \beta$.  
 
\end{abstract}

\textit{Keywords:} Benjamin-Ono equation, dissipative-dispersive effects, Locally and Global well-posedness.
\section{Introduction and main results}
We study the initial value problem (IVP) for the following fractional growth-dissipative Benjamin-Ono (fDBO) equations 
\begin{equation}\label{fDBO}
\left\{
\begin{aligned}
&u_t+\mathcal{H}u_{xx}-(D_x^{\alpha}-D_x^{\beta})u+uu_x=0, \qquad && x\in \mathbb{R} \, (\text{or } x\in \mathbb{T})  \, , \quad t > 0,  \\
&u(x,0)=u_0(x), 
\end{aligned}
\right.
\end{equation}
where $u=u(x,t)$ is a real valued function,  the third and the fourth terms in \eqref{fDBO} will be considered as the growth and dissipation respectively, satisfying $0< \alpha < \beta$, the operator $D_x^s$ is defined via the Fourier transform by $\widehat{D_x^s\varphi}(\xi)=|\xi|^s\widehat{\varphi}(\xi)$ and $\mathcal{H}$ denotes the usual Hilbert transform given by
\begin{equation*}
    \mathcal{H}\varphi(x)=\dfrac{1}{\pi}\,\text{p.v.}\int_{-\infty}^{\infty}\dfrac{\varphi(y)}{x-y}\,dy= \bigl(-i\, \text{sgn}(\xi)\widehat{\varphi}(\xi)\bigr)^{\vee}(x),\quad\text{for}\; \xi \in \mathbb{R},\; \varphi\in \mathcal{S}(\mathbb{R}).
\end{equation*}

When $\alpha=\beta$, fDBO corresponds to the well-known Benjamin-Ono (BO) equation  
derived by Benjamin \cite{benjamin} and Ono \cite{Ono} as a model for long internal gravity waves in deep stratified fluids. The IVP associated to BO equation has been widely studied, see \cite{saut,Iorio,Ponce1991,MoliSautT,KochT,KenigKoenig,TaoBO,Kenig,MoliPeriodBO,molinetPilod} and re\-fe\-rences therein. Several authors have searched the minimal regularity, measured in the Sobolev scale $H^s(\mathbb{R})$, which guarantees that the IVP for BO is {\it locally} or {\it globally wellposed} (LWP and GWP, resp.). We say that an IVP is LWP in a functional space $X$ provided that for every initial data $u_0\in X$ there exists $T=T(\left\|u_0\right\|_{X})>0$ and a unique solution $u\in X_T \subset C([0,T];X)$ of the IVP such that the flow-map data solution is locally continuous from $X$ to $X_T$. If the above properties are true for any $T>0$, we say that the IVP is GWP. Let us recall some of them for the BO equation: in \cite{saut} LWP for $s>3$ was established, in \cite{Iorio} and \cite{Ponce1991} GWP for $s\geq 3/2$, in \cite{TaoBO} GWP when $s\geq 1$, and finally in \cite{Kenig,molinetPilod} GWP when $s\geq 0$ was proven. All these results have been obtained by compactness methods. This is a consequence of the results of Molinet, Saut and Tzvetkov in \cite{MoliSautT} who proved 
for all $s\in \mathbb{R}$ that the flow map $u_0\mapsto u$ is not of class $C^2$ at the origin from $H^s(\mathbb{R})$ to $H^s(\mathbb{R})$. In other words, they showed that one cannot solve the IVP for BO equation by a Picard iterative method implemented on its integral formulation for initial data in the Sobolev space $H^s(\mathbb{R})$, $s\in \mathbb{R}$. 
\\ \\
The fDBO equations \eqref{fDBO} generalize the Chen-Lee (CL) equation 
\begin{equation}\label{cl}
\left\{
\begin{aligned}
&u_t+ \mathcal{H}u_{xx} - (\mathcal{H}u_x+u_{xx}) +uu_x=0, \hspace{0.2cm} x\in \mathbb{R}\, (\text{or } x\in \mathbb{T}) \, , \quad t > 0, \\
&u(x,0)=u_0(x),
\end{aligned}
\right.
\end{equation} 
which corresponds to the case $\alpha=1$ and $\beta=2$ in fDBO. The model \eqref{cl} was first introduced by Chen and Lee in \cite{CL} to describe fluid and plasma turbulence and as a model for internal waves in a two-fluid system. 
\\ \\
When $\alpha=1$ and $\beta=3$,  \eqref{fDBO} becomes the following nonlocal perturbation of the BO (npBO) equation
\begin{equation}\label{npBO}
\left\{
\begin{aligned}
&u_t+\mathcal{H}u_{xx}-(\mathcal{H}u_x+\mathcal{H}u_{xxx})+uu_x=0, \hspace{0.2cm} x\in \mathbb{R}\, (\text{or } x\in \mathbb{T}) \, , \quad t > 0, \\
&u(x,0)=u_0(x).
\end{aligned}
\right.
\end{equation}
\\ \\
Another physically relevant equation within the class \eqref{fDBO} is obtained by choosing $\alpha=1$ and $\beta=4$,
\begin{equation}\label{dBO4}
\left\{
\begin{aligned}
&u_t+\mathcal{H}u_{xx}-(\mathcal{H}u_x-u_{xxxx})+uu_x=0, \hspace{0.2cm} x\in \mathbb{R}\, (\text{or } x\in \mathbb{T}) \, , \quad t > 0, \\
&u(x,0)=u_0(x).
\end{aligned}
\right.
\end{equation}
The models \eqref{cl}, \eqref{npBO} and \eqref{dBO4} have been used in fluids and plasma theory, see (40) in \cite{clq}. 
\\ \\
The form of these equations and the dispersive and dissipative effects involved have motivated us to define the fDBO model, in which we address well-posedness issues in the spaces $H^s(\mathbb{R})$ and $H^s(\mathbb{T})$ for arbitrary $\beta>0$ and $0<\alpha<\beta$. In addition, the structure of \eqref{fDBO} allows us to consider the behavior of solutions as $\alpha \to \beta^{-}$.

Let us now state our results. Initially, we set dissipation order $1<\beta<2$ with growth order $0<\alpha<\beta$. 
Our first consequence determinates LWP by means of a fixed-point argument on some Bourgain type spaces $X^{b,s}$ adapted to the dispersive-dissipative part of fDBO (see \eqref{BourgSpac} below). The advantage of using these spaces lies in the fact that they incorporate dispersion effects which seem to be stronger than those of dissipation and growth previously fixed. 
\begin{theorem}\label{globalwlowdis}
Let $1<\beta < 2$ with $0< \alpha < \beta$ fixed and $u_0 \in H^s(\mathbb{R})$, $s>-\beta/4$. Then for any time $T>0$ there exists a unique solution $u$ of the integral equation \eqref{inteq} in
\begin{align*}
&Z_T=C([0,T];H^s(\mathbb{R}))\cap X_T^{1/2,s}.
\end{align*}
Moreover, the flow map $u_0 \mapsto u(t)$ is smooth from $H^s(\mathbb{R})$ to $Z_T$ and $u$ belongs to $ C((0,T],H^{\infty}(\mathbb{R}))$.
\end{theorem}
Next we consider the case with dissipation order $\beta\geq 2$ and growth order $0<\alpha<\beta$. Here the fDBO behaves as a pure dissipative model and so we can implement techniques relaying mostly on the dissipative term. 
\begin{theorem}\label{globalw}
Let $\beta \geq 2$ and $0<\alpha < \beta$ fixed. Consider $u_0 \in H^s(\mathbb{R})$ where $s>\max\left\{3/2-\beta,-\beta/2\right\}$. Then for any time $T>0$ there exists a unique solution $u$ of the integral equation \eqref{inteq} in
\begin{align*}
&W_T=C([0,T];H^s(\mathbb{R}))\cap Y_T^{s}.
\end{align*}
Moreover, the flow map $u_0 \mapsto u(t)$ is smooth from $H^s(\mathbb{R})$ to $W_T$ and $u$ belongs to $ C((0,T],H^{\infty}(\mathbb{R}))$.
\end{theorem}
The proof of Theorems \ref{globalwlowdis} and \ref{globalw} is based on the ideas in \cite{vento2011,TaokZtheo} and in \cite{Dix} respectively. A main difference in our arguments is the inclusion of the growth term in \eqref{fDBO} which yields additional difficulties due to its iteration with the dissipation. For instance, one has that the operator $D_x^{\alpha}-D_x^{\beta}$ is not completely dissipative since its Fourier symbol $|\xi|^{\alpha}-|\xi|^{\beta}$ is non-negative for frequencies $|\xi|\leq 1$.

\begin{rem} 
 Considering $\beta>1$ fixed, we have that the results in Theorems \ref{globalwlowdis} and \ref{globalw} are independent of the growth term in the sense that the Sobolev regularity attained only depends on the dissipation order $\beta$. Actually, our conclusions are similar to those achieved by Vento in \cite{vento2011} for the dissipative Benjamin-Ono equations (obtained by removing $D_x^{\alpha}$ from  \eqref{fDBO}).
\end{rem}

As it has been determined for other dispersion-dissipative models (see for instance \cite{Pilod,PR,xma}), one may ask for optimally for the results in Theorem \ref{globalwlowdis} and \ref{globalw}  measured by the regularity of the data-solution mapping associated to \eqref{fDBO}. In this direction we provide the following result.

\begin{theorem}\label{illpos}

\begin{itemize}
    \item[(i)] Let $\beta \geq 1$, $0<\alpha<\beta$ and assume that $s<-\beta/2$. Then there does not exist any time $T>0$ such that the Cauchy problem \eqref{fDBO} admits a unique local solution defined on the interval $[0,T]$ and such that the flow-map $u_0\mapsto u$ is $C^2$ differentiable at zero from $H^s(\mathbb{R})$ to $C\left([0,T];H^s(\mathbb{R})\right)$.
   \item[(ii)] Let $\beta \geq 1$, $0<\alpha<\beta$ and assume that $s<\min\left\{3/2-\beta, -\beta/4\right\}$. Then there does not exist any time $T>0$ such that the Cauchy problem \eqref{fDBO} admits a unique local solution defined on the interval $[0,T]$ and such that the flow-map $u_0\mapsto u$ is $C^3$ differentiable at zero from $H^s(\mathbb{R})$ to $C\left([0,T];H^s(\mathbb{R})\right)$.
   \item[(iii)] Let $0< \alpha <\beta <1$,  and $s\in \mathbb{R}$. There does not exist $T>0$ such that the IVP \eqref{fDBO} admits a unique local solution defined on the interval $[0,T]$ and such that the flow map $u_0\mapsto u$ is of class $C^2$ in a neighborhood of the origin from $H^s(\mathbb{R})$ to $H^s(\mathbb{R})$.
\end{itemize}
\end{theorem}
The proof of Theorem \ref{illpos} is inspired by the results in \cite{vento2011} with several modifications dealing with the inclusion of the term $D_x^{\alpha}$. Theorem \ref{illpos} establishes that Theorem \ref{globalwlowdis} and \ref{globalw} are sharp in the sense that the flow map of the IVP \eqref{fDBO} fails to be $C^2$ in $H^s(\mathbb{R})$ for $s<-\beta/2$ and it fails to be $C^3$ in $H^s(\mathbb{R})$ when $s<\min\left\{3/2-\beta,-\beta/4\right\}$. In addition, Theorem \ref{illpos} (iii) determines that whenever $0<\alpha <\beta<1$, we cannot obtain solutions of \eqref{fDBO} via a contraction argument. Finally, we remark that at the end-point $\beta=1$, our proof of Theorem \ref{illpos} (iii) fails. However, Theorem \ref{illpos} (ii) provides ill-posedness in $H^s(\mathbb{R})$ for $s<-1/4$. As a consequence it is still not clear what happens to fDBO when $\beta=1$, $\alpha <\beta$ and $s\geq -1/4$. 
\\ \\
Another interesting property that can be determined by the structure of the \eqref{fDBO} equations is the behavior of solutions as the growth order $\alpha$ converges to the dissipation order $\beta$, i.e., when the effect of dissipation and growth cancel each other. In this respect we have:

\begin{prop}\label{contInsta}
Let $\beta>1$ and $u_0\in H^{s_0}(\mathbb{R})$.  If $s_0>3/2$, then there exist $T>0$ and a function $g\in C([0,T];[0,\infty))$, such that for all $0<\alpha \leq \beta$ there are solutions $u^{\alpha}\in C([0,T];H^{s_0}(\mathbb{R}))$ of \eqref{fDBO} with growth order $\alpha$, dissipation order $\beta$ and initial data $u_0$ such that
    \begin{equation}
        \left\|u^{\alpha}(t)\right\|_{H^{s_0}} \leq g(t), \hspace{0.5cm} t\in [0,T].
    \end{equation}
Moreover, if $s_0>3/2+\max\left\{\beta/2,1\right\}$, then for each $s<s_0-\max\left\{\beta/2,1\right\}$ the mapping
\begin{equation}\label{ContMap}
    \alpha \in (0,\beta] \longmapsto u^{\alpha} \in C([0,T];H^{s}(\mathbb{R}))
\end{equation}
is continuous.
\end{prop}
 The first conclusion of Proposition \ref{contInsta} asserts that there exists a common time $T>0$ at which sufficiently regular solutions of \eqref{fDBO} are uniformly bounded independent of $0<\alpha\leq \beta$. In addition, Proposition \ref{contInsta} establishes some strong convergences between solutions of the growth-dissipation problem $u^{\alpha}$ to solutions of the dispersive Benjamin-Ono equation $u^{\beta}$.
 \\ \\
Concerning the periodic fDBO equations, our conclusions are summarized in the following theorem.
\begin{theorem}\label{periodiccase}
Assume that $\beta>3/2$ with $0<\alpha<\beta$ fixed. Then the conclusion of Theorem \ref{globalw} and Theorem \ref{illpos} part (i) and Proposition \ref{contInsta} still hold in the periodic case with $H^s(\mathbb{R})$ replaced by $H^s(\mathbb{T})$ and $Y_T^s$ replaced by $\tilde{Y}_T^s$. In particular the consequent GWP results established for the periodic fDBO equations are sharp whenever $\beta \geq 3$.  
\end{theorem}

\begin{rem}
Theorem \ref{periodiccase} establishes GWP for the IVP \eqref{fDBO} in $H^s(\mathbb{T})$, with $s>\max\{3/2-\beta , -\beta/2 \}$ if $\beta>3/2$, and for $\beta\geq 1$ it shows that the flow map $u_0\mapsto u$ fails to be $C^2$ in $H^s(\mathbb{T})$ for $s<-\beta /2$. This implies that the GWP result is sharp for $\beta \geq 3$. It is not clear what happens to the IVP associated to the periodic fDBO equations for either $3/2<\beta < 3$ and $-\beta/2\leq s \leq 3/2-\beta$ or for $0<\beta \leq 3/2$ and $s\geq -\beta/2$. 
\end{rem}

Let us now discuss some consequences of our results restricted to the particular cases \eqref{cl}, \eqref{npBO} and \eqref{dBO4}. Here we reprove the conclusions in \cite{P,PR,PR1} for the IVP associated to the CL equation \eqref{cl}. Thus, for this problem, we obtain GWP in $H^s(\mathbb{R})$ and in $H^s(\mathbb{T})$ when $s> -1/2$, we also show that the flow map data-solution for \eqref{cl} fails to be $C^3$ at the origin of $H^s(\mathbb{R})$ when $s<-1/2$, and it lacks $C^2$ regularity at the origin of $H^s(\mathbb{T})$ if $s<-1$. 

Regarding the IVP for the npBO equation \eqref{npBO}, we obtain the same GWP in $H^s(\mathbb{R})$ established in \cite{FRG} and we deduce new global results in periodic Sobolev spaces. More specifically, we deduce GWP in $H^s(\mathbb{R})$ and $H^s(\mathbb{T})$ for $s>-3/2$ and sharp results in the sense that the flow map $u_0\mapsto u$ for npBO fails to be $C^2$ at zero from $H^s(\mathbb{R})$ to $H^s(\mathbb{R})$ or from $H^s(\mathbb{T})$ to $H^s(\mathbb{T})$ when $s<-3/2$.

Finally, in the case of the IVP associated to \eqref{dBO4} which had not been studied before, we deduce GWP in $H^s(\mathbb{R})$ and in $H^s(\mathbb{T})$ when $s>-2$, and we find that its flow map data-solution lacks of $C^2$ regularity at the origin of $H^s(\mathbb{R})$ and $H^s(\mathbb{T})$ when $s<-2$. In this manner, our conclusions for the equations \eqref{npBO} and \eqref{dBO4}  are sharp in both contexts real and periodic. 
\\ 
The organization of the paper is as follows. We begin by introducing some notation and functional spaces to be employed in our arguments. 
In section 2, under the assumptions that $1<\beta<2$ with $0<\alpha<\beta$, we  deduce the crucial bilinear estimates on the $X^{b,s}$ spaces, which ultimately leads to the conclusion of Theorem \ref{globalwlowdis}. In the following section we set $\beta\geq 2$ and $0<\alpha<\beta$ to prove Theorem \ref{globalw}. 
 In the fourth section we show the ill-posedness results stated in Theorem \ref{illpos}. The fifth section is aimed to deduce Proposition \ref{contInsta}. We conclude the paper studying the periodic fDBO equations, that is, we prove Theorem \ref{periodiccase}.

\subsection{Notation and Preliminaries}

The notation we will employ is quite standard. $A\lesssim B$ (for $A$ and $B$ nonnegative) means that there exists $C>0$ independent of $A$ and $B$ such that $A\leq CB$. Similarly define $A\gtrsim B$ and $A\sim B$. Given $p\in [1,\infty]$, we define its conjugate $p'\in [1,\infty]$ from the relation $1=\frac{1}{p}+\frac{1}{p'}$. For such values of $p$, we define the Lebesgue spaces $L^p(\mathbb{R})$ by its norm by $$\left\|f\right\|_{L^p}=\left(\int_{\mathbb{R}} |f(x)|^p\,dx\right)^{1/p},$$
with the usual modification when $p=\infty$. We also consider space-time Lebesgue spaces 
$$\left\|f\right\|_{L^q_t L^p_x}=\left\|\left\|f(\cdot,t)\right\|_{L_x^p}\right\|_{L^{q}_t} \text{ and } \left\|f\right\|_{L^q_T L^p_x}=\left\|\left\|f(\cdot,t)\right\|_{L_x^p}\right\|_{L^{q}_t([0,T])}.$$

The usual Fourier transform is given by $$\widehat{f}(\xi)=\mathcal{F}f(\xi)=\frac{1}{\sqrt{2\pi}}\int_{\mathbb{R}} e^{-i x\cdot \xi} f(x) \, dx.$$
The factor $1/\sqrt{2\pi}$ in the definition of the Fourier transform does not alter our analysis, so will omit it. 

Given $s\in \mathbb{R}$, the $L^2$-based Sobolev spaces $H^s(\mathbb{R})$ are defined by
$$H^s(\mathbb{R})=\left\{f\in S'(\mathbb{R}):\, \left\|f\right\|_{H^s}<\infty\right\},$$
where
$$\left\|f\right\|_{H^s}=\left\|\langle \xi \rangle^{s}\widehat{f}(\xi) \right\|_{L^2}=\left\|J^s u\right\|_{L^2},$$
with $\langle \cdot \rangle=(1+|\cdot|^2)^{1/2}$ and $J^s$ defined by the Fourier symbol $\langle \xi \rangle^s$.

Similarly we define $\dot{H}^s(\mathbb{R})$ by its norm $\left\|f\right\|_{\dot{H}^s}=\left\||\xi|^s\widehat{f}(\xi) \right\|_{L^2}$. 
Recall that for $\lambda>0$,
\begin{equation}\label{scaleineq}
\left\|f(\lambda \cdot)\right\|_{H^s} \leq (\lambda^{-1/2}+\lambda^{s-1/2}) \left\|f\right\|_{H^s} \,  \text{ and }   \left\|f(\lambda \cdot)\right\|_{\dot{H}^s}\sim \lambda^{s-1/2}\left\|f\right\|_{\dot{H}^s}.
\end{equation}


We also consider space-time spaces $H^{b,s}(\mathbb{R}^2)$ endowed with the norm 
    $$\left\|f\right\|_{H^{b,s}}=\left\|\langle\tau \rangle^b \langle\xi \rangle^s \widehat{f}(\xi,\tau) \right\|_{L^2_{\tau}L^2_{\xi}}.$$
Let $U(\cdot)$ be the unitary group in $H^s(\mathbb{R})$, $s\in \mathbb{R}$ associated to the linear Benjamin-Ono equation, {\it i.e.},
\begin{equation*}
    \mathcal{F}_x(U(t)\varphi)(\xi)=e^{-it \xi|\xi|}\widehat{\varphi}(\xi), 
\end{equation*}
for $t\in \mathbb{R}$ and $\varphi \in H^s(\mathbb{R})$.

We denote by $S(\cdot)$ the $H^s(\mathbb{R})$ semigroup generated by the operator $\mathcal{H}\partial_{xx} -(D_x^{\alpha}-D_x^{\beta})$, which is equivalently defined via the Fourier transform by
$$\mathcal{F}_x(S(t)\varphi)(\xi)=e^{-i|\xi|\xi t+(|\xi|^{\alpha}-|\xi|^{\beta})t}\widehat{\varphi}(\xi), \hspace{0.5cm} t\geq 0.$$
We extend $S(\cdot)$ to a linear operator on the whole real axis by setting  
$$\mathcal{F}_x(S(t)\varphi)(\xi)=e^{-i|\xi|\xi t+(|\xi|^{\alpha}-|\xi|^{\beta})|t|}\widehat{\varphi}(\xi), \hspace{0.5cm} t\in \mathbb{R}.$$
When the dissipation order satisfy $1<\beta<2$ with order growth $0<\alpha<\beta$, we introduce the function space $X^{b,s}$ in the sense of Bourgain \cite{Bourgain} and Molinet and Ribaud \cite{moli-Rib} to be the completion of the Schwartz space $S(\mathbb{R}^2)$ on $\mathbb{R}^2$ endowed with the norm
\begin{equation}\label{BourgSpac}
    \left\|u\right\|_{X^{b,s}}=\left\|\langle i(\tau+\xi|\xi|) -(|\xi|^{\alpha}-|\xi|^{\beta})\rangle^b\langle \xi \rangle^s \widehat{u}(\xi,\tau)\right\|_{L^2(\mathbb{R}^2)},
\end{equation}
or equivalently,
\begin{equation*}
    \left\|u\right\|_{X^{b,s}}=\left\|\langle |\tau+\xi|\xi|| +||\xi|^{\alpha}-|\xi|^{\beta}|\rangle^b\langle \xi \rangle^s \widehat{u}(\xi,\tau) \right\|_{L^2(\mathbb{R}^2)}.
\end{equation*}
For $T>0$, we consider the localized spaces $X_T^{b,s}$ endowed with the norm
\begin{equation}
    \left\|u\right\|_{X^{b,s}_T}=\inf\left\{\left\|w\right\|_{X^{b,s}}\, : \, w(t)=u(t) \text{ on } [0,T]\right\}.
\end{equation}
Next, we consider the restrictions $\beta \geq 2$ with $0<\alpha <\beta$. In this case, we can rely on pure dissipative methods to deduce well-posedness results. Thus, given $s\in \mathbb{R}$ and $0<t\leq T\leq 1$ fixed, we denote by
$$Y^s_T=\left\{u\in C([0,T];H^s(\mathbb{R})): \, \left\|u\right\|_{Y_T^s}<\infty \right\},$$
where
\begin{equation} \label{aeq5}
    \left\|u\right\|_{Y_T^s}:= \sup_{t\in(0,T]} \left(\left\|u(t)\right\|_{H^s}+t^{|s|/\beta}\left\|u(t)\right\|_{L^2}\right).
\end{equation}
Note that when $s\geq 0$, $Y_T^s=C([0,T];H^s(\mathbb{R}))$ and $\left\|u\right\|_{Y_T^s}\sim \left\|u\right\|_{L^{\infty}_T H^s}$.

We mainly work on the integral formulation of \eqref{fDBO} denoted by
\begin{equation}\label{inteq}
    u(t)=S(t) u_0-\frac{1}{2}\int_{0}^{t} S(t-\tau) \partial_x u^2(\tau)\, d\tau, \hspace{0.2cm} t\geq 0,
\end{equation}
valid for any sufficiently regular solution. When the dissipation $\beta <2$ it will be convenient to replace the local-in-time integration \eqref{inteq} with a global-in-time truncated equation. Let $\psi$ be a cutoff function  such that
\begin{equation*}
    \psi \in C_0^{\infty}(\mathbb{R}), \hspace{0.2cm}  \supp(\psi) \subset [-2,2], \hspace{0.2cm} \psi \equiv 1 \, \text{ on } [-1,1],
\end{equation*}
and set $\psi_T(\cdot)=\psi(\cdot/T)$ for all $T>0$. Thus we can replace \eqref{inteq} on time interval $[0,T]$, $T<1$ by the equation
\begin{equation}\label{inteqglob}
    u(t)=\psi(t)\left[S(t) u_0-\frac{\chi_{\mathbb{R}^{+}}(t)}{2}\int_{0}^{t} S(t-\tau) \partial_x(\psi_T^2(\tau)u^2(\tau))\, d\tau\right], \hspace{0.2cm} t\geq 0.
\end{equation}


\section{Well-Posedness case \texorpdfstring{$1<\beta< 2$}{}.}

In this section we establish Theorem \ref{globalwlowdis} when the dissipative order lies in $1<\beta<2$ and the growth order $0<\alpha<\beta$. In this case, we see that the dispersive part of \eqref{fDBO} has an important role to obtain low regularity solutions. Consequently, we will apply a contraction argument on the integral equation \eqref{inteq} on the $X_T^{b,s}$ spaces, which consider these effects. The main ingredient to apply this technique is the derivation of a key bilinear estimate (see Proposition \ref{bilineaest} below).

\subsection{Linear estimates}

In this part we estimate the operator $\psi(\cdot)S(\cdot)$ as well as the linear operator $L$ defined by
\begin{equation*}
    L: \, f \mapsto \, \chi_{\mathbb{R}^{+}}(t) \psi(t) \int_0^t S(t-\tau)f(\tau)\, d\tau .
\end{equation*}
We first study the  action of the semigroup $\left\{S(t)\right\}_{t\geq 0}$ on $H^s(\mathbb{R})$, $s\in \mathbb{R}$. 
\begin{prop}\label{aprop1} 
Let $0<\alpha < \beta $ and
 \begin{equation}\label{expsg}
    \psi_{\alpha,\beta}(t)=\exp\left(\left(\frac{2\alpha}{\beta}\right)^{\frac{\alpha}{\beta-\alpha}}\frac{(\beta-\alpha)}{\beta} t\right), \qquad t\in \mathbb{R}.
\end{equation}
Consider $\delta \geq0$ and $s\in \mathbb{R}$. Then for all $t>0$ it follows
\begin{equation}\label{aeq14}
        \left\|S(t) \phi\right\|_{H^{s+\delta}} \lesssim \psi_{\alpha,\beta}(t) \,\left(1+t^{-\delta/\beta}\right)\left\|\phi\right\|_{H^s},
\end{equation}
where $\phi\in H^s(\mathbb{R})$ and the implicit constant depends on $\delta$ and $\beta$. Moreover, the map $t\mapsto S(t) \phi$ belongs to $C((0,\infty);H^{s+\delta}(\mathbb{R}))$.

\end{prop}

\begin{proof}
Noting that for all $0<\alpha < \beta$ and $t>0$,
\begin{equation*}
    \left\|e^{(|\xi|^{\alpha}-|\xi|^{\beta}/2)t}\right\|_{L^{\infty}}=\exp\left(\left(\frac{2\alpha}{\beta}\right)^{\frac{\alpha}{\beta-\alpha}}\frac{(\beta-\alpha)}{\beta} t\right)=\psi_{\alpha,\beta}(t),
\end{equation*}
we derive the bound
\begin{equation}\label{semigb}
    |e^{(|\xi|^{\alpha}-|\xi|^{\beta})t}| \leq \psi_{\alpha,\beta}(t)e^{-|\xi|^{\beta}t/2}.
\end{equation}
From this and setting $w=t^{1/\beta}\xi$, we see that
\begin{equation}\label{aeq15}
    \begin{aligned}
     \left\|\langle \xi\rangle^{\delta} e^{(|\xi|^{\alpha}-|\xi|^{\beta})t}\right\|_{L^{\infty}}\leq \psi_{\alpha,\beta}(t)\left\|\langle t^{-1/\beta}w\rangle^{\delta}e^{-|w|^{\beta} /2}\right\|_{L^{\infty}}.
    \end{aligned}
\end{equation}
Since $$(1+t^{-2/\beta}|w|^2)^{\delta/2}\lesssim 1+t^{-\delta/\beta}|w|^\delta,$$
we find
\begin{equation*}
    \begin{aligned}
     \left\|S(t)\phi\right\|_{H^{s+\delta}}&\leq \left\|\langle \xi \rangle^{\delta}  e^{(|\xi|^{\alpha}-|\xi|^{\beta})t}\right\|_{L^{\infty}}\left\|\phi\right\|_{H^{s}}\\
     & \lesssim \psi_{\alpha,\beta}(t) \,\left(1+t^{-\delta/\beta}\right)\left\|\phi\right\|_{H^s}.
    \end{aligned}
    \end{equation*}
This establish \eqref{aeq14}. The continuity of the map $t\mapsto S(t)\phi$ is deduced arguing as in \cite[Proposition 2.2]{biag}. 
\end{proof}

\begin{lemma}\label{BLILEMMA1}
For all $s\in \mathbb{R}$ and all $\varphi\in H^s(\mathbb{R})$
\begin{equation}
    \left\|\psi(t)S(t)\varphi \right\|_{X^{1/2,s}} \lesssim \left\|\varphi \right\|_{H^s}.
\end{equation}
\end{lemma}

\begin{proof}
By definition of the $\left\|\cdot\right\|_{X^{1/2,s}}$-norm, we find the following upper-bound
\begin{equation}\label{BLE1}
    \begin{aligned}
     \left\|\psi(t)S(t)\varphi \right\|_{X^{1/2,s}} &\lesssim \left\|\langle \xi \rangle^s  \widehat{\varphi}(\xi) \left\|\langle \tau \rangle^{1/2} \mathcal{F}_t(g_{\xi}(t)) \right\|_{L^2_\tau(\mathbb{R})} \right\|_{L^2_{\xi}(\mathbb{R})} \\
     &\hspace{0.3cm}+ \left\|\langle \xi \rangle^s \langle |\xi|^{\alpha}-|\xi|^{\beta} \rangle^{1/2}  \widehat{\varphi}(\xi) \left\|g_{\xi}(t) \right\|_{L^2_t(\mathbb{R})} \right\|_{L^2_{\xi}(\mathbb{R})}.
    \end{aligned}
\end{equation}
where we have set $g_{\xi}(t):=\psi(t)e^{(|\xi|^{\alpha}-|\xi|^{\beta})|t|}$. Therefore, in view of \eqref{BLE1}, it is enough to estimate  $\left\|g_{\xi}\right\|_{H^{b}_t}$ for $b\in \left\{0,1/2\right\}$. First assume that $|\xi|\geq 2^{\frac{1}{\beta-\alpha}}$, then
\begin{equation}
    \begin{aligned}
     \left\|g_{\xi}\right\|_{H^{b}_t} &= \left\|\langle \tau \rangle^{b}\, \widehat{\psi}\ast\mathcal{F}_t(e^{(|\xi|^{\alpha}-|\xi|^{\beta})|t|})\right\|_{L^2} \\
     &\lesssim \left\|\langle \tau\rangle^{b}\widehat{\psi}(\tau)\right\|_{L^1_{\tau}}\left\|e^{(|\xi|^{\alpha}-|\xi|^{\beta})|t|}\right\|_{L^2_t}+ \left\|\widehat{\psi}(\tau)\right\|_{L^1_{\tau}}\left\|e^{(|\xi|^{\alpha}-|\xi|^{\beta})|t|}\right\|_{H^{1/2}_t},
    \end{aligned}
\end{equation}
so that \eqref{scaleineq} yields to
\begin{equation}\label{BLE2}
    \left\|g_{\xi}\right\|_{H^{b}_t}\lesssim (||\xi|^{\alpha}-|\xi|^{\beta}|^{b-1/2}+||\xi|^{\alpha}-|\xi|^{\beta}|^{-1/2}).
\end{equation}
Note that $||\xi|^{\alpha}-|\xi|^{\beta}|\sim |\xi|^{\beta}$ when $|\xi|\geq 2^{\frac{1}{\beta-\alpha}}$, and so $\left\|g_{\xi}\right\|_{H^{1/2}_t} \lesssim 1$ in this case. On the other hand, when $|\xi|< 2^{\frac{1}{\beta-\alpha}}$,
\begin{equation*}
    \begin{aligned}
     \left\|g_{\xi}\right\|_{H^{b}_t}=\left\|\psi(t)e^{(|\xi|^{\alpha}-|\xi|^{\beta})|t|}\right\|_{H^b_t} \lesssim \sum_{n=0}^{\infty} \frac{||\xi|^{\alpha}-|\xi|^{\beta}|^{n}}{n!}\left\||t|^n \psi(t)\right\|_{H^b}.
    \end{aligned}
\end{equation*}
Since $n\geq 1$, $\left\||t|^n \psi(t)\right\|_{H^b}\leq \left\||t|^n \psi(t)\right\|_{H^1} \lesssim n$, it follows 
\begin{equation}\label{BLE3}
    \begin{aligned}
     \left\|g_{\xi}\right\|_{H^{b}_t} \lesssim 2^{\frac{\beta }{\beta-\alpha}}\sum_{n=1}^{\infty} \frac{2^{\frac{\beta (n-1)}{\beta-\alpha}}}{(n-1)!} \lesssim 1.
    \end{aligned}
\end{equation}
Hence combining \eqref{BLE2} and \eqref{BLE3} we arrive at
\begin{equation*}
    \left\|g_{\xi}\right\|_{H^{b}_t} \lesssim \langle |\xi|^{\alpha}-|\xi|^{\beta} \rangle^{b-1/2}, 
\end{equation*}
for $b \in \left\{0,1/2\right\}$. The desire estimate now follows substituting the above inequality in \eqref{BLE1}.
\end{proof}
A simple inspection to the proof of Lemma \ref{BLILEMMA1} shows that our arguments were divided according to the regions  $||\xi|^{\alpha}-|\xi|^{\beta} |\sim |\xi|^{\beta}$ and $||\xi|^{\alpha}-|\xi|^{\beta}|\sim |\xi|^{\alpha}$. In this manner, following this same procedure and adapting the ideas in  \cite{P} and \cite{moli-Rib}, it is not difficult to deduce the following conclusions:
\begin{prop}\label{BLIPROP2}
Let $s\in \mathbb{R}$
\begin{itemize}
    \item[(i)] There exists $C>0$ such that, for all $v\in S(\mathbb{R}^2)$,
    \begin{equation}
    \begin{aligned}
     &\left\|\chi_{\mathbb{R}^{+}}(t)\psi(t)\int_0^t S(t-\tau)v(\tau)\, d\tau\right\|_{X^{1/2,s}} \\
     & \hspace{3cm}\leq C \left[\left\|v\right\|_{X^{-1/2,s}}+\left(\int \langle \xi \rangle^{2s}\left(\int \frac{|\widehat{w}(\tau)|}{\langle i\tau -(|\xi|^{\alpha}-|\xi|^{\beta})\rangle}\, d\tau\right)^2 d\xi\right)^{1/2}\right].
    \end{aligned}
    \end{equation}
    
    \item[(ii)] For any $0<\delta <1/2$ there exists $C_{\delta}>0$ such that for all $v\in X^{-1/2+\delta,s}$
    \begin{equation}\label{FORCINESTM}
        \left\|\chi_{\mathbb{R}^{+}}(t)\psi(t)\int_0^t S(t-\tau)v(\tau)\, d\tau\right\|_{X^{1/2,s}} \leq C_{\delta} \left\|v\right\|_{X^{-1/2+\delta,s}}.
    \end{equation}
\end{itemize}
\end{prop}

\begin{prop}\label{propconti}
Let $s\in \mathbb{R}$ and $\delta>0$. For all $T>0$ and each $f\in X^{-1/2+\delta,s}$,
\begin{equation}\label{BLE9}
    t \mapsto \int_0^t S(t-\tau)f(\tau)\, d \tau \in C([0,T];H^{s+\beta\delta}(\mathbb{R})).
\end{equation}
Moreover, if $(f_n)$ is a sequence with $f_n \underset{n\to \infty}{\rightarrow} 0$ in $X^{-1/2+\delta,s},$ then
\begin{equation}\label{BLE9.0}
    \left\|\int_0^t S(t-\tau)f_n(\tau)\, d \tau\right\|_{L^{\infty}([0,T];H^{s+\beta \delta})} \to 0.
\end{equation}
\end{prop}

We conclude this subsection with the following result which will be used in the proof of Theorem \ref{globalwlowdis} and is proved following a similar reasoning in \cite[Lemma 2.5]{GINIBRE}

\begin{lemma}\label{lemmalocesti}
Let $s\in \mathbb{R}$. For any $\epsilon>0$ and $T\in(0,1]$, the following estimate follows
\begin{equation}
    \left\|\psi_T u\right\|_{X^{1/2,s}} \lesssim T^{-\epsilon} \left\|u\right\|_{X^{1/2,s}}.
\end{equation}
\end{lemma}


\subsection{Bilinear estimates}

To deduce the crucial bilinear result we will apply the dyadic block estimates deduced by Vento in \cite{vento2011}. 

We first introduce some notation and results to be employed in our arguments which are based on Tao's $[k;Z]$-multiplier theory \cite{TaokZtheo}. Let $Z$ be any abelian additive group with an invariant measure $d\eta$. For any integer $k\geq 2$ we define the hyperplane
\begin{equation*}
    \Gamma_k(Z)=\left\{(\eta_1,\dots,\eta_k)\in Z^{k}\, : \, \eta_1+\dots+\eta_k=0\right\},
\end{equation*}
which is endowed with the measure
\begin{equation*}
    \int_{\Gamma_k(Z)} f=\int_{Z^{k-1}} f(\eta_1,\dots,\eta_{k-1},-(\eta_1+\dots+\eta_{k-1}))\, d\eta_1\dots d\eta_{k-1}.
\end{equation*}
A $[k;Z]$-multiplier is defined to be any function $m: \Gamma_k(Z)\rightarrow \mathbb{C}$. The norm of the multiplier $\left\|m\right\|_{[k;Z]}$ is defined to be the best constant such that the inequality 
\begin{equation}
    \left|\int_{\Gamma_k(Z)}m(\eta)\prod_{j=1}^k f_j(\eta_j)\right|\leq \left\|m\right\|_{[k;Z]}\prod_{j=1}^{k}\left\|f_j\right\|_{L^2(\mathbb{R})},
\end{equation}
holds for all test function $f_1,\dots,f_k$ on $Z$. In other words,
\begin{equation}
    \left\|m\right\|_{[k;Z]}=\sup_{\substack{f_j\in S(Z) \\ \left\|f_j\right\|_{L^2(Z)}\leq 1}}\left|\int_{\Gamma_k(Z)}m(\eta)\prod_{j=1}^k f(\eta_j)\right|.
\end{equation}

Following the notation in \cite{TaokZtheo}, capitalized variables such as $N_j, L_j, H$ are presumed to be
dyadic, i.e. these variables range over numbers of the form $2^l$ for $l\in \mathbb{Z}$. Let $N_1,N_2,N_3>0$. we define the quantities $N_{\max}\geq N_{med}\geq N_{min}$ to be the maximum, median, and minimum of $N_1,N_2, N_3$ respectively.
Similarly, we define  $L_{\max}\geq L_{med}\geq L_{min}$ whenever $L_1, L_2, L_3 >0$. The quantities
$N_j$ will measure the magnitude of frequencies of our waves, while $L_j$ measures how closely our waves approximate a free solution. 

We adopt the following summation conventions. Any summation of the form $L_{max} \sim \dots$ is a sum over the three dyadic variables $L_1,L_2,,L_3 \gtrsim 1$, thus for instance
\begin{equation*}
    \sum_{L_{max}\sim H}:=\sum_{L_1,L_2,L_3 \gtrsim 1 : L_{max}\sim  H}
\end{equation*}
Similarly, any summation of the form $N_{max}\sim \dots$ sum over the three dyadic variables $N_1, N_2, N_3 > 0$, hence
\begin{equation*}
    \sum_{L_{max}\sim H}:=\sum_{L_1,L_2,L_3 \gtrsim 1 : L_{max}\sim  H}
\end{equation*}

Due to the nonlinear term in \eqref{fDBO}, we will consider $[3;\mathbb{R}\times \mathbb{R}]$-multipliers and the variables will be set as $\eta=(\xi,\tau)$ with the usual Lebesgues measure $d\eta=d\xi\, d \tau$. We let
\begin{equation*}
    h_0(\theta)=-\theta |\theta|, \hspace{0.5cm} \lambda_j=\tau_j-h_0(\xi_j), \hspace{0.5cm} j=1,2,3,
\end{equation*}
and the resonance function
\begin{equation*}
    h(\xi)=h_0(\xi_1)+h_0(\xi_2)+h_0(\xi_3), \hspace{0.2cm} \xi=(\xi_1,\xi_2,\xi_3).
\end{equation*}
By a dyadic decomposition of the variables $\xi_j,\lambda_j,h(\xi)$, we will lead to estimate
\begin{equation}\label{multip}
    \left\|X_{N_1,N_2,N_3,H,L_1,L_2,L_3}\right\|_{[3;\mathbb{R}\times \mathbb{R}]},
\end{equation}
where $X_{N_1,N_2,N_3,H,L_1,L_2,L_3}$ is the multiplier 
\begin{equation}\label{BBLR1}
    X_{N_1,N_2,N_3,H,L_1,L_2,L_3}=\chi_{|h(\xi)|\sim  H}\prod_{j=1}^3 \chi_{|\xi_j|\sim N_j}\chi_{|\lambda_j|\sim L_j}.
\end{equation}
From the identities
\begin{equation}
    \xi_1+\xi_2+\xi_3=0
\end{equation}
and
\begin{equation}
    \lambda_1+\lambda_2+\lambda_3+h(\xi)=0,
\end{equation}
on the support of the multiplier, we see that \eqref{BBLR1} vanish unless
\begin{equation}\label{BLOCK1}
    N_{max}\sim N_{med}
\end{equation}
and
\begin{equation}\label{BLOCK2}
    L_{max}\sim \max(H,L_{med}).
\end{equation}
As a consequence it is not difficult to deduce the following result.
\begin{lemma}
On the support of $X_{N_1,N_2,N_3,H,L_1,L_2,L_3}$, one has
\begin{equation}\label{BLOCK3}
    H\sim N_{max}N_{min}.
\end{equation}
\end{lemma}

Our arguments depend on the following dyadic blocks estimates.
\begin{lemma}(\cite{vento2011})\label{lemmabil1}
Let $N_1,N_2,N_3,H,L_1,L_2,L_3>0$ satisfying \eqref{BLOCK1}, \eqref{BLOCK2} and \eqref{BLOCK3}.
\begin{itemize}
    \item[(i)] In the high modulation case $L_{max}\sim L_{med} \gg H$, we have
    \begin{equation}\label{HMCASE}
         \eqref{multip}\lesssim L_{min}^{1/2}N^{1/2}_{min}.
    \end{equation}
    \item[(ii)] In the low modulation case $L_{max}\sim H$,
    \begin{itemize}
        \item[(a)] ((++) coherence) if $N_{max}\sim N_{min}$ then 
           \begin{equation}\label{COHEPP}
        \text{ \eqref{multip}}\lesssim L_{min}^{1/2}L^{1/4}_{med}.
    \end{equation}
    \item[(b)] ((+-) coherence) If $N_2\sim N_3 \gg N_1$ and $H\sim L_1 \gtrsim L_2,L_3$, then for any $\gamma>0$ 
    \begin{equation}\label{COHEPN}
        \eqref{multip}\lesssim L_{min}^{1/2}\min(N_{min}^{1/2},N_{max}^{1/2-1/2\gamma}N_{min}^{-1/2 \gamma}L^{1/2 \gamma}_{med}.
    \end{equation}
    Similar for any permutations of indexes $\left\{1,2,3\right\}$.
    \item[(c)] In all other cases, the multiplier \eqref{multip} vanishes.   
    \end{itemize}
\end{itemize}
\end{lemma}

The main goal of this section is to derive the following key bilinear estimate.

\begin{theorem}\label{bilineaest}
Let $1<\beta \leq 2$, $0<\alpha < \beta$ and $s>-\beta/4$. For all $T>0$, there exists $\delta,\nu>0$ such that for all $u,v \in X^{1/2,s}$ with compact support (in time) in $[-T,T]$
\begin{equation}
    \left\|\partial_x(uv)\right\|_{X^{-1/2+\delta,s}}\lesssim T^{\nu} \left\|u\right\|_{X^{1/2,s}}\left\|v\right\|_{X^{1/2,s}}.
\end{equation}
\end{theorem}

Actually, we will consider the following bilinear estimate, which is a direct consequence of Theorem \ref{bilineaest}, together with the triangle inequality
\begin{equation}
\forall s \geq s_c^{+}, \hspace{0.2cm} \langle \xi \rangle^{s_c^{+}} \leq \langle \xi \rangle^{s_c^{+}}\langle \xi_1 \rangle^{s-s_c^{+}}+\langle \xi \rangle^{s_c^{+}}\langle \xi-\xi_1 \rangle^{s-s_c^{+}}.
\end{equation}

\begin{prop}\label{bilineaestcri}
 Given $s_{c}^{+}>-\beta/4$, there exist $\nu,\delta>0$ such that for any $s\geq s_c^{+}$ and $u,v \in X^{1/2,s}$ with compact support in $[-T,T]$,
\begin{equation}
    \left\|\partial_x(uv)\right\|_{X^{-1/2+\delta,s}}\lesssim T^{\nu}\left( \left\|u\right\|_{X^{1/2,s_c^{+}}}\left\|v\right\|_{X^{1/2,s}}+\left\|u\right\|_{X^{1/2,s}}\left\|v\right\|_{X^{1/2,s_c^{+}}}\right).
\end{equation}
\end{prop}

The next lemma gives the contraction factor $T^{\nu}$ in our estimates (see \cite{otani}).
\begin{lemma}\label{bilinesttime}
Let $f\in L^2(\mathbb{R}^2)$ with compact on $[-T, T]$. For any $\theta>0$, there exists $\nu=\nu(\theta)$ such that
\begin{equation}
    \left\|\mathcal{F}^{-1}\left(\frac{\widehat{f}(\xi,\tau)}{\langle \tau+\xi|\xi| \rangle^{\theta}}\right)\right\|_{L^2_{\xi \tau}} \lesssim T^{\nu}\left\|f\right\|_{L^2_{xt}}.
\end{equation}
\end{lemma}

\begin{proof}[Proof of Theorem \ref{bilineaest}] By duality and Lemma \ref{bilinesttime}, it is enough to show 
\begin{equation}\label{BBLR2}
    \left\|\frac{\xi_3\langle \xi_3 \rangle^s \langle \xi_1 \rangle^{-s}\langle \xi_2 \rangle^{-s}}{\langle |\lambda_1|+||\xi_1|^{\alpha}-|\xi_1|^{\beta}|\rangle^{1/2}\langle |\lambda_2|+||\xi_2|^{\alpha}-|\xi_2|^{\beta}|\rangle^{1/2}\langle|\lambda_3|+||\xi_3|^{\alpha}-|\xi_3|^{\beta}|\rangle^{1/2-\delta}}\right\|_{[3;\mathbb{R}\times \mathbb{R}]} \lesssim 1.
\end{equation}
By a dyadic decomposition of the variables $\xi_j, \lambda_j$ and $h(\xi_j)$, we may assume $|\xi_j|\sim N_j$, $|\lambda_j|\sim L_j$ and $|h(\xi)|\sim H$. By the translation invariance of the $[k;Z]$-multiplier (see \cite[Lemma 3.4]{TaokZtheo}) we may restrict the multiplier to the region $L_j \gtrsim 1$ and $N_{max} \gtrsim 1$. Furthermore, to consider the particular structure of the operator $D_x^{\alpha}-D_x^{\beta}$ in the frequency domain, we define
\begin{equation*}
    I(N_j):=\inf_{|\xi_j|\sim N_j} ||\xi_j|^{\alpha}-|\xi_j|^{\beta}|,
\end{equation*}
and analogously we set $I(N)$. Noting that if $|\xi_j|\geq 2^{\frac{1}{\beta-\alpha}}$, $||\xi_j|^{\alpha}-|\xi_j|^{\beta}|=|\xi_j|^{\beta}-|\xi_j|^{\alpha}\geq |\xi_j|^{\beta}/2$, we deduce
\begin{equation}\label{BBLR3}
    \left\{\begin{aligned}
     &I(N_j)\lesssim \max\left\{N_j^{\alpha}, N_j^{\beta}\right\}, &&\text{ when } N_j\lesssim 1, \\
     &I(N_j)\sim N_j^{\beta}, &&\text{ when } N_j \gg 1.
    \end{aligned}\right.
\end{equation}
In view of \eqref{BBLR3}, we shall consider separately the cases $N\gg 1$ and $N\sim1$. Gathering the results in \cite{TaokZtheo} (Schur's test, comparison principle and orthogonality), it is deduced that \eqref{BBLR2} is bounded by one of the following inequalities
\begin{equation}\label{BBLR4}
    \begin{aligned}
       \sum_{N_{max}\sim N_{med} \sim N}\sum_{L_1,L_2,L_3 \gtrsim 1}&  \frac{N_3\langle N_3\rangle^s \langle N_1\rangle^{-s}\langle N_2\rangle^{-s}}{\langle L_1+I(N_1)\rangle^{1/2}\langle L_2+I(N_2)\rangle^{1/2}\langle L_2+I(N_2)\rangle^{1/2-\delta}} \\
     &\times \left\|X_{N_1,N_2,N_3,L_{max},L_1,L_2,L_3}\right\|_{[3;\mathbb{R}\times \mathbb{R}]},
    \end{aligned}
\end{equation}
and
\begin{equation}\label{BBLR5}
    \begin{aligned}
      \sum_{N_{max}\sim N_{med} \sim N}\sum_{L_{max}\sim L_{med}}\sum_{H \ll L_{max}}&  \frac{N_3\langle N_3\rangle^s \langle N_1\rangle^{-s}\langle N_2\rangle^{-s}}{\langle L_1+I(N_1)\rangle^{1/2}\langle L_2+I(N_2)\rangle^{1/2}\langle L_2+I(N_2)\rangle^{1/2-\delta}} \\
     &\times \left\|X_{N_1,N_2,N_3,H,L_1,L_2,L_3}\right\|_{[3;\mathbb{R}\times \mathbb{R}]}.
    \end{aligned}
\end{equation}
Therefore, we are reduced to bound the above expressions for all $N \gtrsim 1$. We will divide our arguments according to Lemma \ref{lemmabil1}.

\underline{\bf High modulation case.} Here $L_{max}\sim L_{med}\gg H$, and so we must show that $\eqref{BBLR5}\lesssim 1$. In fact, this result follows under the weaker assumption that $s>-1/2$. It easily seen 
\begin{equation*}
    N_3\langle N_3 \rangle^s \langle N_1 \rangle^{-s}\langle N_2 \rangle^{-s} \lesssim \langle N_{min} \rangle^{-s}N_{max},
\end{equation*}
then \eqref{HMCASE} yields 
\begin{equation}\label{BBLR6}
    \begin{aligned}
\eqref{BBLR5} \lesssim \sum_{N_{max}\sim N_{med} \sim N}\sum_{L_{max}\sim L_{med} \gtrsim NN_{min}}  \frac{\langle N_{min} \rangle^{-s}N L_{min}^{1/2}N_{min}^{1/2}}{\langle L_1+I(N_1)\rangle^{1/2}\langle L_2+I(N_2)\rangle^{1/2}\langle L_2+I(N_2)\rangle^{1/2-\delta}}. \\
    \end{aligned}
\end{equation}
 Note that for fixed $N_{max}$, $N_{min}$ the sum over $H\ll L_{max}$ in \eqref{BBLR5} is finite, since $H\sim N_{max}N_{min}$ in virtue of \eqref{BLOCK3}. Assuming that $N \sim 1$, we have $N_{max}\sim N_{med}\sim 1$, $N_{min}\lesssim 1$ and $L_{max}\sim L_{med} \gtrsim N_{min}$. Estimating under these assumptions one gets
\begin{equation}
    \langle L_1+I(N_1)\rangle^{1/2}\langle L_2+I(N_2)\rangle^{1/2}\langle L_3+I(N_3)\rangle^{1/2-\delta} \gtrsim L_{min}^{1/2}L_{max}^{\delta}N_{min}^{1/2-\delta}.
\end{equation}
Thus, it follows 
\begin{equation}
    \begin{aligned}
     \eqref{BBLR5} &\lesssim \sum_{N_{min}\lesssim 1}\sum_{L_{max}\sim L_{med} \gtrsim 1} \frac{\langle N_{min}\rangle^{-s}L_{min}^{1/2}N_{min}^{1/2}}{L_{min}^{1/2}L_{max}^{\delta}N_{min}^{1/2-\delta}} \\
     &\lesssim \sum_{N_{min}\lesssim 1} N_{min}^{\delta} \lesssim 1.
    \end{aligned}
\end{equation}
Now suppose that $N\gg 1$. In this case, $N_{max}\sim N_{med}\sim N \gg 1$, $L_{max}\sim L_{min} \gtrsim NN_{min}$, and so from \eqref{BBLR3} we obtain
\begin{equation}
    \langle L_1+I(N_1)\rangle^{1/2}\langle L_2+I(N_2)\rangle^{1/2}\langle L_3+I(N_3)\rangle^{1/2-\delta} \gtrsim L_{min}^{1/2}L_{max}^{\delta}(NN_{min}+N^{\beta})^{1/2-\delta}(NN_{min})^{1/2-\delta}.
\end{equation}
Then,
\begin{equation}
    \begin{aligned}
     \eqref{BBLR5} &\lesssim \sum_{N_{max}\sim N}\sum_{L_{max}\sim L_{med} \gtrsim N N_{max}} \frac{\langle N_{min}\rangle^{-s}N L_{min}^{1/2}N_{min}^{1/2}}{L_{min}^{1/2}L_{max}^{\delta}(NN_{min}+N^{\beta})^{1/2-\delta}(NN_{min})^{1/2-\delta}} \\
     &\lesssim \sum_{N_{min}> 0} \frac{\langle N_{min}\rangle^{-s}N N_{min}^{1/2}}{(NN_{min}+N^{\beta})^{1/2-\delta}(NN_{min})^{1/2-\delta}}.
    \end{aligned}
\end{equation}
Consequently,
\begin{equation*}
\begin{aligned}
  \eqref{BBLR5} &\lesssim \sum_{N_{min} \lesssim 1} \frac{N N_{min}^{1/2}}{N^{\beta/2-\beta\delta}(NN_{min})^{1/2-\delta}}+ \sum_{N_{min} \gtrsim 1} \frac{ N_{min}^{1/2-s}N}{(NN_{min})^{1-2\delta-\epsilon}N^{\beta \epsilon}}\\
    &\lesssim \sum_{N_{min} \lesssim 1} N_{min}^{\delta} N^{(1-\beta)/2+(\beta+1)\delta}+\sum_{N_{min} \gtrsim 1} N_{min}^{-1/2-s+2\delta+\epsilon} N^{2\delta-\epsilon(\beta-1)} \\
    & \lesssim 1,
\end{aligned}
\end{equation*}
 which holds when $\beta>1$, $\delta \ll 1$, $\epsilon=2\delta/(\beta-1)>0$ and $s>-1/2$. This completes the estimate $\eqref{BBLR5} \lesssim 1$. 

\underline{ \bf Low modulation case: (++) coherence}. Now we will show that $\eqref{BBLR4}\lesssim 1$, assuming that $L_{\max}\sim H$ and the contribution \eqref{COHEPP}. In this case, $N_{min}\sim N_{mid}\sim N_{max}\sim N$. If $N\gg 1$, \eqref{BBLR3} gives
\begin{equation}
\begin{aligned}
     \langle L_1+I(N_1)\rangle^{1/2}&\langle L_2+I(N_2)\rangle^{1/2}\langle L_3+I(N_3)\rangle^{1/2-\delta} \\
     &\gtrsim L_{min}^{1/2}L_{max}^{\delta}(L_{med}+N^{\beta})^{1/2}(L_{max}+N^{\beta})^{1/2-2\delta}\\
     &\gtrsim L_{min}^{1/2}L_{max}^{\delta}L_{med}^{1/4}N^{\beta/4}L_{max}^{1/2-2\delta}.
\end{aligned}
\end{equation}
Consequently, since $L_{max}\sim N^2$, we deduce
\begin{equation}
    \begin{aligned}
       \eqref{BBLR4} &\lesssim \sum_{L_{max} \sim N^2} \frac{N^{1-s}L_{min}^{1/2}L_{med}^{1/4}}{L_{min}^{1/2}L_{max}^{\delta}L_{med}^{1/4}N^{\beta/4}L_{max}^{1/2-2\delta}} \\
       &\lesssim \frac{N^{1-s}}{N^{\beta/4}N^{1-4\delta}}\\
       &\lesssim N^{-s-\beta/4+4\delta} \lesssim 1,
    \end{aligned}
\end{equation}
when $s>-\beta/4$ and $\delta \ll 
1$. Now, assume that $N\sim 1$, thus we find
\begin{equation}
\begin{aligned}
     \langle L_1+I(N_1)\rangle^{1/2}&\langle L_2+I(N_2)\rangle^{1/2}\langle L_3+I(N_3)\rangle^{1/2-\delta} \\
     &\gtrsim L_{min}^{1/2}L_{max}^{\delta}L_{med}^{1/4}L_{max}^{1/2-2\delta}.
\end{aligned}
\end{equation}
A similar reasoning as in the previous case leads to
\begin{equation}
     \begin{aligned}
       \eqref{BBLR4} 
       \lesssim N^{1-s}N^{-1+2\delta}\sim N^{-s+2\delta} \sim 1.
    \end{aligned}
\end{equation}
\underline{ \bf Low modulation case: (+-) coherence.} We will show that $\eqref{BBLR4}\lesssim 1$, when \eqref{COHEPN} holds. By symmetry it suffices to treat the cases
\begin{equation}\label{restric}
    \begin{aligned}
     N_1\sim N_2 \gg N_3, &\text{ and } H\sim L_3 \gtrsim L_1,L_2, \\
    N_2\sim N_3 \gg N_1, & \text{ and } H\sim L_1 \gtrsim L_2,L_3.
    \end{aligned}
\end{equation}
    In the first case, by letting $\gamma=1$ in \eqref{COHEPN}, it is easily seen 
    \begin{equation}
    \begin{aligned}
     \eqref{multip} \lesssim L_{min}^{1/2}\min(N_{min}^{1/2},N_{min}^{-1/2 }L^{1/2}_{med}) \lesssim L_{min}^{1/2}L_{med}^{1/4}.
    \end{aligned}
    \end{equation}
Then, when $N\gg 1$,
\begin{equation}
\begin{aligned}
     \langle L_1+I(N_1)\rangle^{1/2}&\langle L_2+I(N_2)\rangle^{1/2}\langle L_3+I(N_3)\rangle^{1/2-\delta} \\
     &\gtrsim L_{min}^{1/2}L_{max}^{\delta}N^{\beta/4}L_{med}^{1/4}(N_{max}N_{min})^{1/2-2\delta},
\end{aligned}
\end{equation}
so that
\begin{equation}
    \begin{aligned}
         \eqref{BBLR4} &\lesssim \sum_{N_3 >0} \sum_{L_{max}\sim NN_3} \frac{N_3 \langle N_3 \rangle^s N^{-2s}L_{min}^{1/2}L_{med}^{1/4}}{L_{min}^{1/2}L_{max}^{\delta}N^{\beta/4}L_{med}^{1/4}(NN_{3})^{1/2-2\delta}}\\
         &\lesssim \sum_{N_3 >0} \frac{N_3 \langle N_3 \rangle^s N^{-2s}}{N^{\beta/4+1/2-2\delta}N_{3}^{1/2-2\delta}} \\
         &\lesssim \sum_{N_3 >0} 
         N_3^{1/2+2\delta} \langle N_3 \rangle^s N^{-2s-\beta/4-1/2+2\delta}.
  \end{aligned}
\end{equation}
Since $-2s-\beta/4-1/2+2\delta<0$ the above inequality allow us to deduce
\begin{equation}
    \begin{aligned}
         \eqref{BBLR4} &\lesssim \sum_{N_3 \lesssim 1}
         N_3^{1/2+2\delta}+\sum_{N_3 \gtrsim 1}
         N_3^{-s-\beta/4+4\delta}\lesssim 1,
  \end{aligned}
\end{equation}
which holds for $\delta \ll 1$ and $s>-\alpha/4$. When $N\sim 1$ one gets
\begin{equation}
\begin{aligned}
     \langle L_1+I(N_1)\rangle^{1/2}&\langle L_2+I(N_2)\rangle^{1/2}\langle L_3+I(N_3)\rangle^{1/2-\delta} \\
     &\gtrsim L_{min}^{1/2}L_{max}^{\delta}L_{med}^{1/2}(N_{max}N_{min})^{1/2-2\delta},
\end{aligned}
\end{equation}
so since $L_{min} \gtrsim 1$, $N\sim 1$ and $\delta \ll 1$ we arrive at
\begin{equation}
    \begin{aligned}
         \eqref{BBLR4} &\lesssim \sum_{N_3 \lesssim 1} \sum_{L_{max}\sim N_3} \frac{N_3 \langle N_3 \rangle^s N^{-2s}L_{min}^{1/2}L_{med}^{1/4}}{L_{min}^{1/2}L_{max}^{\delta}L_{med}^{1/2}(NN_{3})^{1/2-2\delta}}\\
         & \lesssim \sum_{N_3} N_3^{1/2+2\delta}\langle N_3 \rangle^{s} N^{-1/2-2s+2\delta}\\
         &\lesssim \sum_{N_3 \lesssim 1} N_3^{1/2+2\delta} \lesssim 1.
  \end{aligned}
\end{equation}

Next we consider the second restriction in \eqref{restric}, i.e., $N_2\sim N_3 \gg N_1,  \text{ and } H\sim L_1 \gtrsim L_2,L_3$. Let $0<\gamma \ll 1$. We first suppose that $N_{min}^{1/2}\lesssim N_{max}^{1/2-1/2\gamma}N_{min}^{-1/2\gamma}L_{med}^{1/2\gamma}$, which shows $L_{med}\gtrsim N_{max}^{1-\gamma}N_{min}^{1+\gamma}$. One has the following lower bound 
\begin{equation}\label{BBLR6.5} 
\begin{aligned}
     \langle L_1+I(N_1)\rangle^{1/2}&\langle L_2+I(N_2)\rangle^{1/2}\langle L_3+I(N_3)\rangle^{1/2-\delta} \\
     &\gtrsim L_{min}^{1/2}L_{max}^{\delta}L_{max}^{1/2-\delta}\langle L_{med}+I(N)\rangle^{1/2-\delta}.
\end{aligned}
\end{equation}
Thus, \eqref{COHEPN} and \eqref{BBLR6.5} imply
\begin{equation}\label{BBLR7}
    \begin{aligned}
             \eqref{BBLR4} &\lesssim \sum_{N_1 \lesssim N} \sum_{L_{max}\sim NN_1} \frac{ \langle N_1 \rangle^{-s} N L_{min}^{1/2}N_1^{1/2}}{L_{min}^{1/2}L_{max}^{\delta}L_{max}^{1/2-\delta}\langle L_{med}+I(N)\rangle ^{1/2-\delta}}\\
             &\lesssim \sum_{N_1 \lesssim N}  \frac{ \langle N_1 \rangle^{-s} N N_1^{1/2}}{(NN_{1})^{1/2-\delta}\langle N^{1-\gamma}N_1^{1+\gamma}+I(N)\rangle^{1/2-\delta}}\\
             &\lesssim \sum_{N_1 \lesssim N}  \frac{ N_1^{\delta}\langle N_1 \rangle^{-s} N^{1/2+\delta}}{\langle N^{1-\gamma}N_1^{1+\gamma}+I(N)\rangle^{1/2-\delta}}.
  \end{aligned}
\end{equation}
When $N\sim 1$, \eqref{BBLR7} shows
\begin{equation}
    \begin{aligned}
             \eqref{BBLR4} 
             &\lesssim \sum_{N_1 \lesssim 1}   N_1^{\delta}\langle N_1 \rangle^{-s} \lesssim 1.
  \end{aligned}
\end{equation}
Suppose that $N\gg 1$, so $I(N)\sim N^{\beta}$. Then when $N_1 \lesssim 1$, \eqref{BBLR7} implies
\begin{equation}
    \begin{aligned}
             \eqref{BBLR4} 
             &\lesssim \sum_{N_1 \lesssim 1}N_1^{\delta}N^{(1-\beta)/2+\delta(1+\beta)}\lesssim 1,
  \end{aligned}
\end{equation}
for $\delta \ll 1$ and $\beta >1$. If $N_1\gtrsim 1$, from \eqref{BBLR7} we find
\begin{equation}
    \begin{aligned}
             \eqref{BBLR4} 
             &\lesssim \sum_{N_1 \gtrsim 1}  \frac{ N_1^{-s+\delta} N^{1/2+\delta}}{(N^{1-\gamma}N_1^{1+\gamma})^{1/2-\delta-\epsilon}N^{\beta \epsilon}} \\
             &\lesssim \sum_{N_1 \gtrsim 1}   N_1^{-s-1/2+(1+\gamma)(\delta+\epsilon)+\delta-\gamma/2} N^{\gamma(1/2-\delta)+2\delta-\epsilon(\beta-1+\gamma)}\lesssim 1
  \end{aligned}
\end{equation}
for $\delta,\gamma \ll 1$, $s>-1/2$ and $\epsilon=(2\delta+\gamma(1/2-\delta))/(\beta-1+\gamma)>0$. 

Now, we consider the case $N_{min}^{1/2} \gtrsim N_{max}^{1/2-1/2\gamma}N_{min}^{-1/2\gamma}L_{med}^{1/2\gamma}$, i.e., $L_{med} \lesssim N_{max}^{1-\gamma}N_{min}^{1+\gamma}$. First we assume that $N\sim1$. Since the right-hand side of \eqref{BBLR6.5} is  bounded below by $L_{min}^{1/2}L_{max}^{\delta}$ for $0<\delta<1/2$, one gets
\begin{equation}
    \begin{aligned}
             \eqref{BBLR4} &\lesssim \sum_{N_1 \lesssim 1} \sum_{L_{max}\sim NN_1} \frac{ \langle N_1 \rangle^{-s} N L_{min}^{1/2}N^{1/2-1/2\gamma}N_1^{-1/2\gamma}L_{med}^{1/2\gamma}}{L_{min}^{1/2}L_{max}^{\delta}} \\
                 &\lesssim \sum_{N_1 \lesssim 1} \langle N_1 \rangle^{-s} NN_1^{1/2}\lesssim 1.
    \end{aligned}
\end{equation}
Now we assume that $N\gg 1$. Inequality \eqref{BBLR6.5} allow us to deduce 
\begin{equation}
    \begin{aligned}
             \eqref{BBLR4} &\lesssim \sum_{N_1>0} \sum_{L_{max}\sim NN_1} \frac{ \langle N_1 \rangle^{-s} N L_{min}^{1/2}N^{1/2-1/2\gamma}N_1^{-1/2\gamma}L_{med}^{1/2\gamma}}{L_{min}^{1/2}L_{max}^{\delta}L_{max}^{1/2-\delta}\langle L_{med}+N^{\beta}\rangle ^{1/2-\delta}} \\
                 &\lesssim \sum_{N_1 >0}\sum_{L_1,L_2,L_3\gtrsim 1:L_{med} \lesssim N^{1-\gamma}N_1^{1+\gamma}}  \frac{N_1^{-1/2\gamma -1/2+\delta}\langle N_1 \rangle^{-s}N^{1-1/2\gamma +\delta}L_{med}^{1/2\gamma}}{L_{max}^{\delta}\langle L_{med}+N^{\beta}\rangle ^{1/2-\delta}}.
    \end{aligned}
\end{equation}
So when $N_1\lesssim 1$, $\delta \ll 1$ and $\beta>1$, we have
\begin{equation}
    \begin{aligned}
             \eqref{BBLR4} 
                 &\lesssim \sum_{N_1 \lesssim 1}  N_1^{-1/2\gamma -1/2+\delta}N^{1-1/2\gamma +\delta}N^{-\beta/2+\beta \delta}(N^{1-\gamma}N_1^{1+\gamma})^{1/2\gamma} \\
                  &\lesssim \sum_{N_1 \lesssim 1}  N_1^{\delta}N^{(1-\beta)/2+\delta(1+\beta)}\lesssim 1.
    \end{aligned}
\end{equation}
When $N_1 \gtrsim 1$,
\begin{equation}
    \begin{aligned}
             \eqref{BBLR4} 
                 &\lesssim \sum_{N_1 \gtrsim 1}  N_1^{-s-1/2-1/2\gamma+\delta}N^{1-1/2\gamma +\delta}(N^{1-\gamma}N_1^{1+\gamma})^{1/2\gamma-1/2+\delta+\epsilon}N^{-\beta \epsilon} \\
                  & \lesssim \sum_{N_1 \gtrsim 1}  N_1^{-s-1/2+(1+\gamma)(\delta+\epsilon)+\delta-\gamma/2}N^{\gamma(1/2-\delta)+2\delta-\epsilon(\beta-1+\gamma)}\lesssim 1.
    \end{aligned}
\end{equation}
This concludes the estimates regarding the low modulation case. Thus, the proof of Theorem \ref{bilineaest} is now completed.
\end{proof}

We are in condition to deduce Theorem \ref{globalwlowdis}. 

\begin{proof}[Proof of Theorem]
We divide the proof in three main steps.
\\ \\
1. \emph{Local well-posedness}. Let $u_0 \in H^s(\mathbb{R})$, $s>-\beta/4$. We consider the Banach space
$$Z=\left\{u\in X^{1/2,s}\,:\, \left\|u\right\|_{Z}=\left\|u\right\|_{X^{1/2,s_c^{+}}}+\gamma\left\|u\right\|_{X^{1/2,s}}<\infty \right\},$$ 
where $s_{c}^{+}\in (-\beta/4,\min\left\{0,s\right\})$ and $\gamma$ is defined for all nontrivial $u_0$,
\begin{equation*}
    \gamma=\frac{\left\|u_0\right\|_{H^{s_c^{+}}}}{\left\|u_0\right\|_{H^{s}}}.
\end{equation*}
 Given $0<T\leq 1$ to be chosen later, we define the integral map
 \begin{equation}\label{contramap}
     \Psi(u)=\psi(t)\left[S(t) u_0-\frac{\chi_{\mathbb{R}^{+}}(t)}{2}\int_{0}^t S(t-\tau) \partial_x(\psi_T(\tau)u(\tau))^2 \ d\tau \right],
 \end{equation}
for each $u \in Z$. In view of Lemmas \ref{BLILEMMA1}, \ref{lemmalocesti}, estimate \eqref{FORCINESTM}, Theorem \ref{bilineaest} and Proposition \ref{bilineaestcri} there exist some constants $c,\nu>0$ such that 
\begin{align}
\left\|\Psi(u)\right\|_{Z} &\leq c\left(\left\|u_0\right\|_{H^{s_c^{+}}}+\gamma\left\|u_0\right\|_{H^{s}}\right)+cT^{\nu}\left\|u\right\|_{Z}^2, \label{BBLR8} \\
\left\|\Psi(u)-\Psi(v)\right\|_{Z} & \leq c T^{\nu}\left\|u-v\right\|_{Z}\left\|u+v\right\|_{Z}, \label{BBLR9}
\end{align}
for all $u, v \in Z$. Thus, recalling the definition of $\gamma$, we consider $0<T\leq \min\left\{1,(16c^2\left\|u_0\right\|_{H^{s_c^{+}}})^{-1/\nu}\right\}$. Then \eqref{BBLR8} and \eqref{BBLR9} imply that $\Psi$ is a contraction on the ball $\left\{u\in Z : \left\|u \right\|_{Z}\leq 4c\left\|u_0\right\|_{H^{s_c^{+}}} \right\}$. Consequently the fixed-point Theorem assures the existence of a solution $u\in X^{1/2,s}$ of \eqref{inteq} on the time interval $[0,T]$, with $u(0)=u_0$. 
\\ \\
The continuity with respect to the initial data follows directly from Proposition \ref{propconti}. Moreover, the above contraction argument yields uniqueness of solution to the truncated
integral equation \eqref{inteqglob}. The proof of uniqueness for the integral equation \eqref{inteq} can be derived following the same arguments in \cite{moli-Rib} and \cite{vento2011}.
\\ \\
\emph{2. Regularity}. Now we establish that $ u\in C\left((0,T],H^{\infty}(\mathbb{R})\right)\cap X_{T}^{1/2,s}$ and the flow map data solution is smooth. Indeed, in view of Proposition \ref{aprop1}, $S(\cdot)u_0 \in C([0,\infty);H^{s}(\mathbb{R}))\cap C([0,\infty);H^{\infty}(\mathbb{R}))$. Then it follows from Theorem \ref{bilineaest}, Proposition \ref{propconti} and the local well-posedness that 
$$u\in C\left([0,T];H^s(\mathbb{R})\right)\cap C\left((0,T];H^{s+\beta \delta}(\mathbb{R})\right),$$ 
where $T=T(\left\|u_0\right\|_{H^{s_c^{+}}})$. Thus, an inductive argument, the uniqueness result and the fact that the time of existence of solutions depends uniquely on the $H^{s_c^{+}}(\mathbb{R})$-norm of the initial data yield
$$u\in C\left([0,T];H^s(\mathbb{R})\right)\cap C\left((0,T];H^{\infty}(\mathbb{R})\right).$$
The smoothness of the flow-map is a consequence of the implicit
function theorem (see for instance \cite[Remark 3]{bekiranov1996}).
\\ \\
 \emph{3. Global well-posedness}. We define
$$T^*=\sup\left\{T>0: \exists !\text{ solution of \eqref{inteq} in }  C\left([0,T];H^s(\mathbb{R})\right)\cap X_T^{1/2,s}\right\}.$$
Let $u \in C\left([0,T^*);H^s(\mathbb{R})\right)\cap C\left((0,T^*);H^{\infty}(\mathbb{R})\right)$ be the local solution of the integral equation associated to \eqref{fDBO} on the maximal interval $[0,T^*)$. We will prove that $T^*<\infty$ implies a contradiction. Since $u$ is smooth by the above step, we have that this function solves \eqref{fDBO} in a classical sense. Therefore, we can multiply \eqref{fDBO} by $u$ and integrating over $\mathbb{R}$ to obtain
\begin{equation*}
    \frac{1}{2}\frac{d}{dt}\left\|u(t)\right\|_{L^2}^2-\left\|D^{\alpha/2}u(t)\right\|_{L^2}^2+\left\|D^{\beta/2}u(t)\right\|_{L^2}^2=0,
\end{equation*}
given that $0<\alpha < \beta$ the above expression shows
 \begin{equation*}  
   \frac{1}{2}\frac{d}{dt}\left\|u(t)\right\|_{L^2}^2\leq \left\|u(t)\right\|_{L^2}^2.
 \end{equation*}
 Let $t_0\in(0,T^{\ast})$ fixed. Integrating the above inequality between $t_0$ and $t$ and applying Gronwall's inequality to the resulting expression one gets
\begin{equation*}
    \left\|u(t)\right\|_{L^2} \leq \left\|u(t_0)\right\|_{L^2}e^{(T^{\ast}-t_0)}\equiv M, \hspace{0.5cm} \forall t\in [t_0,T^{\ast}).
\end{equation*}
Recalling that $s_{c}^{+}\leq 0$, it follows
\begin{equation*}
    \left\|u(t)\right\|_{H^{s_c^{+}}} \leq  M, \hspace{0.5cm} \forall t\in [t_0,T^{\ast}).
\end{equation*}
Since the time of existence $T(\cdot)$ is a nonincreasing function of the $H^{s_c^{+}}(\mathbb{R})$-norm, there exists  a time $\widetilde{T}>0$, such that for all $\tilde{u}_0 \in H^s(\mathbb{R})$ with $\left\|\tilde{u}_0 \right\|_{H^{s_c^{+}}}\leq M$, there exists a function $\tilde{u}\in C([0,\widetilde{T}];H^s(\mathbb{R}))\cap X_{T'}^{1/2,s}$ solution of the integral equation \eqref{inteq} with $\widetilde{u}(0)=\tilde{u}_0$. Let $0<\epsilon <\min\left\{\widetilde{T},(T^{\ast}-t_0)\right\}$, applying this result to $\tilde{u}_0=u(T^*-\epsilon)$, we define
\begin{equation}
v(t)=\begin{cases} 
      u(t), & \text{ when } \hspace{0.6cm} 0 \leq t \leq T^*-\epsilon, \\
     \tilde{u}(t-T^*+\epsilon), &\text{ when } \hspace{0.6cm} T^*-\epsilon \leq t \leq T^*+\widetilde{T}-\epsilon. \\
   \end{cases}
\end{equation}
Hence, $v(t)$ is a solution of the integral equation \eqref{inteq} on $[0,T^*+\widetilde{T}-\epsilon]$ with initial data $u_0$. Clearly, $T^*+\widetilde{T}-\epsilon >T^*$, which leads to a contradiction to the definition of $T^*$. This proves the global result.

\end{proof}

\section{Well-Posedness case \texorpdfstring{$\beta\geq 2$}{}.}

This section is devoted to prove local and global well-posedness for the equation \eqref{fDBO} when the dissipation order satisfies $\beta \geq 2$ and growth order $0< \alpha <\beta$. Our approach is based on the methods introduced in \cite{Dix}, which rely on the dissipation of the equation. Mainly, the strategy to obtain local existence is to construct a contraction mapping from the integral equation \eqref{inteq} acting on the Banach spaces $Y_T^s$ defined by \eqref{aeq5}. 
\\ \\
We first recall the results in Proposition \ref{aprop1} where it was established
\begin{equation}
    \begin{aligned}
     \left\|S(t)\phi\right\|_{H^s} \lesssim \psi_{\alpha,\beta}(t)\left\|\phi\right\|_{H^s},
    \end{aligned}
\end{equation}
for all $\phi \in H^s(\mathbb{R})$, $s\in \mathbb{R}$ and with
\begin{equation}
     \psi_{\alpha,\beta}(t)=\exp\left(\left(\frac{2\alpha}{\beta}\right)^{\frac{\alpha}{\beta-\alpha}}\frac{(\beta-\alpha)}{\beta} t\right), \qquad t\in \mathbb{R}.
\end{equation}
   
   To evaluate the action of the integral equation \eqref{inteq} on the $Y_T^s$ spaces, we require the following proposition.

\begin{prop}\label{aprop1.1}
Let $\beta >3/2$, $0< \alpha <\beta$ fixed and $\psi_{\alpha,\beta}$ given by \eqref{expsg}. 

\begin{itemize}
\item[(i)] For all $s \geq 0$ and $t>0$
\begin{equation}\label{aeq14.1}
    \left\||\xi|^s e^{(|\xi|^{\alpha}-|\xi|^{\beta})t}\right\|_{L^2}\lesssim \psi_{\alpha,\beta}(t) \, t^{-s/\beta-1/2\beta}.
\end{equation}
    \item[(ii)] Let $s \in \mathbb{R}$. Then for any $0<t\leq 1$ it follows that
\begin{equation}\label{ae10}
\left\||\xi|\langle \xi \rangle^{s}e^{(|\xi|^{\alpha}-|\xi|^{\beta})t}\right\|_{L^2} \lesssim \psi_{\alpha,\beta}(t)\,t^{-r/2}
\end{equation}
for all $r>\max\left\{(3+2s)/\beta,0\right\}$.
\end{itemize}
\end{prop}

\begin{proof}
In view of inequality \eqref{semigb} and changing variables by $w=t^{1/\beta}\xi$, we find 
\begin{equation}\label{aeq16}
    \begin{aligned}
     \left\| |\xi|^s e^{(|\xi|^{\alpha}-|\xi|^{\beta})t}\right\|_{L^2}\leq \psi_{\alpha,\beta}(t)\left\||w|^{s}e^{-|w|^{\beta} /2}\right\|_{L^2}\lesssim  t^{-s/\beta-1/2\beta}.
    \end{aligned}
\end{equation}
This establish (i).
Next, we show \eqref{aeq16} when $0<t\leq 1$. We recall the inequality 
\begin{equation}\label{aeq9}
t^re^{- |\xi|^{\beta} t}\leq \left(\frac{r}{|\xi|^{\beta}}\right)^{r}e^{-r},
\end{equation}
which is valid for all $\xi \neq 0$ and $r>0$. Then, dividing in low and high frequencies, from \eqref{aeq9}, we deduce
\begin{equation*}
\begin{aligned}
\psi_{\alpha,\beta}(-2t)\left\||\xi|\langle \xi \rangle^{s}e^{(|\xi|^{\alpha}-|\xi|^{\beta})t}\right\|_{L^2}^2 \lesssim \int_{|\xi|\leq 1}|\xi|^2(1+|\xi|^2)^{s} \, d\xi+t^{-r}\int_{|\xi| > 1}|\xi|^{2+2s-\beta r} \, d\xi \lesssim t^{-r},
\end{aligned}
\end{equation*}
where $r> \max\left\{(3+2s)/\beta,0\right\}$. This completes the proof of \eqref{ae10}.
\end{proof}

Now we can estimate the integral equation \eqref{inteq} on the spaces $Y_T^s$.

\begin{prop}\label{aprop2} 
Let $0<\alpha < \beta$, $0<  T\leq 1$ and $\psi_{\alpha,\beta}$ defined by \eqref{expsg}.
\begin{itemize}
    \item[(i)] Let $s<0$ and $\phi\in H^s(\mathbb{R})$. Then
$$ \left\|S(t) \phi\right\|_{Y_T^s} \lesssim \psi_{\alpha,\beta}(T) \left\|\phi\right\|_{H^s}.$$
    \item[(ii)] Let $\beta>3/2$ and $s\geq 0$. Then for all $u,v\in C([0,T];H^s(\mathbb{R}))$ it follows
\begin{equation}
    \left\|\int_0^t S(t-\tau) \partial_x(u v)(\tau)\, d\tau \right\|_{L^{\infty}_TH^s_x} \lesssim  \psi_{\alpha,\beta}(T)\, T^{\frac{1}{2\beta}(2\beta-3)}  \left\|u \right\|_{L^{\infty}_TH^s_x} \left\|v \right\|_{L^{\infty}_TH^s_x}.
    \end{equation}
 \item[(iii)] Assume that $\beta > 3/2$ and $\max\left\{3/2-\beta,-\beta/2\right\}<s<0$. Given $u,v\in Y_T^s$ it follows
\begin{equation}
    \left\|\int_0^t S(t-\tau) \partial_x(u v)(\tau)\, d\tau \right\|_{Y_T^s} \lesssim \psi_{\alpha,\beta}(T)\, T^{\frac{1}{2\beta}(2\beta-r\beta+4s)}  \left\|u \right\|_{Y_T^s} \left\|v \right\|_{Y_T^s},
\end{equation}
for some $\max\left\{(3+2s)/\beta,0\right\}< r < 2(\beta+2s)/\beta$.
\end{itemize}
\end{prop}
\begin{proof}
Part (i) is a direct consequence of \eqref{aeq14} (with $\delta=0$) and using that $0\leq t \leq T\leq 1$. To deduce (ii), we consider $s\geq 0$ and the inequality
\begin{equation*}
    \left\|\langle\xi \rangle^{s}\mathcal{F}(u^2(\tau))(\xi)\right\|_{L^{\infty}}\lesssim \left\|u(\tau)\right\|_{H^s}^2\lesssim  \left\|u\right\|_{L^{\infty}_T H^s}^2.
\end{equation*}
Thus, in view of the above estimate and \eqref{aeq14.1},
\begin{equation}\label{aeq4}
    \begin{aligned}
     \left\|\int_0^t S(t-\tau) \partial_x(u v)(\tau)\, d\tau \right\|_{H^s} & \leq \int_0^t  \psi_{\alpha,\beta}(\tau)\left\||\xi|e^{(|\xi|^{\alpha}-|\xi|^{\beta})(t-\tau)}\right\|_{L^2} \left\|\langle \xi\rangle^s\mathcal{F}(uv(\tau))(\xi)\right\|_{L^{\infty}}\, d \tau \\
      & \lesssim \psi_{\alpha,\beta}(T) \left( \int_0^t \tau^{-\frac{3}{2\beta}} \, d \tau \right)\left\|u\right\|_{L_T^{\infty}H^s}\left\|v\right\|_{L_T^{\infty}H^s} \\
      & \lesssim \psi_{\alpha,\beta}(T) T^{\frac{1}{2\beta}(2\beta-3)} \, \left\|u\right\|_{L_T^{\infty}H^s}\left\|v\right\|_{L_T^{\infty}H^s}.
    \end{aligned}
\end{equation}
This shows (ii). To deduce (iii), we consider $s < 0$. The definition of the norm on the space $Y_T^s$ yields the following estimate
\begin{equation}\label{aeq1}
    \left\|u(\tau)v(\tau)\right\|_{L^1} \leq \frac{\left\|u\right\|_{Y_T^s}\left\|v\right\|_{Y_T^s}}{\tau^{\frac{2|s|}{\beta}}},
\end{equation}
so that in view of \eqref{ae10} and performing the change of variables $\sigma=\tau/t$, we deduce
\begin{equation}\label{aeq2}
    \begin{aligned}
      \left\|\int_0^t S(t-\tau) \partial_x(u v)(\tau)\, d\tau \right\|_{H^s} & \leq \int_0^t \psi_{\alpha,\beta}(\tau)\left\||\xi|\langle\xi \rangle^{s}e^{(|\xi|^{\alpha}-|\xi|^{\beta})(t-\tau)}\right\|_{L^2} \left\|u(\tau)v(\tau)\right\|_{L^1}\, d \tau \\
      & \leq \psi_{\alpha,\beta}(T) \int_0^t  \frac{\left\||\xi|\langle\xi \rangle^{s}e^{(|\xi|^{\alpha}-|\xi|^{\beta})(t-\tau)}\right\|_{L^2}}{\tau^{\frac{2|s|}{\beta}}}\, d \tau \left\|u\right\|_{Y_T^s}\left\|v\right\|_{Y_T^s} \\
      &\lesssim \psi_{\alpha,\beta}(T) \,T^{\frac{1}{2\beta}(2\beta-r\beta+4s)}\left(\int_0^1 (1-\sigma)^{-r/2}\sigma^{\frac{-2|s|}{\beta}}\,d \sigma \right) \left\|u\right\|_{Y_T^s}\left\|v\right\|_{Y_T^s},
    \end{aligned}
\end{equation}
where $r>\max\left\{(3+2s)/\beta,0\right\}$ and $0\leq t\leq T$. Arguing in a similar manner, we have for all $0\leq t \leq T$ that
\begin{equation}\label{aeq3}
    \begin{aligned}
      t^{\frac{|s|}{\beta}}\left\|\int_0^t S(t-\tau) \partial_x(u v)(\tau)\, d\tau \right\|_{L^2} & \leq t^{\frac{|s|}{\beta}} \int_0^t \psi_{\alpha,\beta}(\tau) \left\||\xi|e^{(|\xi|^{\alpha}-|\xi|^{\beta})(t-\tau)}\right\|_{L^2}\left\|u(\tau)v(\tau)\right\|_{L^1}\, d \tau \\
      & \lesssim \psi_{\alpha,\beta}(T)\,t^{\frac{|s|}{\beta}} \int_0^t \frac{\left\||\xi|e^{(|\xi|^{\alpha}-|\xi|^{\beta})(t-\tau)}\right\|_{L^2}}{\tau^{\frac{2|s|}{\beta}}}\, d \tau \, \left\|u\right\|_{Y_T^s}\left\|v\right\|_{Y_T^s} \\
      & \lesssim  \psi_{\alpha,\beta}(T) T^{\frac{1}{2\beta}(2\beta+2s-3)}\left(\int_0^1 (1-\sigma)^{-\frac{3}{2\beta}}\sigma^{\frac{-2|s|}{\beta}}\,d \sigma \right) \left\|u\right\|_{Y_T^s}\left\|v\right\|_{Y_T^s}.
    \end{aligned}
\end{equation}
Consequently the right-hand side of inequalities \eqref{aeq2} and \eqref{aeq3} impose the conditions
$$\beta>3/2, \hspace{0.2cm} \text{ and } \hspace{0.2cm} s> \max\left\{3/2-\beta,-\beta/2\right\},$$ 
with $\max\left\{(3+2s)/\beta,0\right\}<r<2(\beta+2s)/\beta$. This remark completes the proof of Proposition \ref{aprop2}.
\end{proof}

\begin{prop}\label{aprop3} 
Let $\beta > 3/2$, $0<\alpha < \beta$, $s>\max\left\{3/2-\beta,-\beta/2\right\}$ and $0< T \leq 1$. Then there exists $\delta=\delta(s,\beta)>0$ such that the application 
$$t\rightarrow \int_{0}^t S(t-\tau) \partial_x(u^2)(\tau)\ d\tau $$
is in $C\left([0,T];H^{s+\delta}(\mathbb{R})\right)$, for every $u\in Y_T^s$. 
\end{prop}
\begin{proof}
The proof is similar to that in \cite[Proposition 4]{Pilod}. 
\end{proof}

We are in condition to prove Theorem \ref{globalw}. 

\begin{proof}[Proof of Theorem \ref{globalw}]

We consider $\beta \geq 2$ and $0<\alpha <\beta$ fixed. Let $u_0 \in H^s(\mathbb{R})$ with $s>\max\left\{3/2-\beta,-\beta/2\right\}$, we define the integral map
$$\tilde{\Psi}(u)=S(t) u_0-\frac{1}{2}\int_{0}^t S(t-\tau) \partial_x(u^2(\tau)) \ d\tau,$$
for each $u \in Y_T^s$. By Proposition \ref{aprop2} there exists a positive constant $c=c(\beta,\alpha)$ such that 
\begin{align}
\left\|\tilde{\Psi}(u)\right\|_{Y_T^s} &\leq c\left(\left\|u_0\right\|_{H^s}+T^{g_{\beta}(s)}\left\|u\right\|_{Y_T^s}^2\right), \label{IN11} \\
\left\|\tilde{\Psi}(u)-\tilde{\Psi}(v)\right\|_{Y_T^s} & \leq c T^{g_{\beta}(s)}\left\|u-v\right\|_{Y_T^s}\left\|u+v\right\|_{Y_T^s}, \label{IN12}
\end{align}
for all $u, v \in Y_T^s$, $0<T\leq 1$ and $s>\max\left\{3/2-\beta,-\beta/2\right\}$. Here the map $g_{\alpha}(s)$ is defined according to Proposition \ref{aprop2}, i.e., $g_{\beta}(s)=\frac{1}{2\beta}(2\beta-3)$ for all $s\geq 0$, and $g_{\beta}(s)=\frac{1}{2\beta}(2\beta-r\beta+2s)$, when $\max\left\{3/2-\beta,-\beta/2\right\}<s<0$, for some fixed $r$ such that $\max\left\{(3+2s)/\beta,0\right\}<r<2(\beta+2s)/\beta$. 

We consider $R=2c\left\|u_0 \right\|_{H^s}$ and $0<T\leq \min \left\{1,\left(4c R \right)^{-\frac{1}{g_{\beta}(s)}} \right\}$. Then estimates \eqref{IN11} and \eqref{IN12} imply that $\tilde{\Psi}$ is a contraction on the complete metric space $\left\{u\in Y_T^s : \left\|u \right\|_{Y_T^s}\leq R \right\}$.  Therefore, the fixed-point Theorem implies the existence of a solution $u$ to the integral equation \eqref{inteq}. \\ \\
The continuity with respect to the initial data is deduced following the same arguments in \cite{PR1}. To verify uniqueness, we let $u,v \in Y_T^s$ solutions of equation \eqref{inteq} on the time interval $[0,T]$ with the same initial data $u_0\in H^s(\mathbb{R})$. Then, arguing as in the deduction of Proposition \ref{aprop2}, we have that there exists a constant $c=c(\alpha,\beta,s)$, such that for all $0<T_1\leq T_2\leq T$,
\begin{equation}\label{WP3}
\begin{aligned}
\left\|(u-v)\chi_{\left\{t\geq  T_1\right\}}\right\|_{Y_{T_2}^s}\leq c K (T_2-T_1)^{g_{\beta}(s)}\left\|u-v\right\|_{Y_{T_2}^s},
\end{aligned}
\end{equation}
where $K:=\left\|u\right\|_{Y_{T}^s}+\left\|v\right\|_{Y_{T}^s}$. Thus, taking $T_2 \in \bigl(0,(cK)^{-\frac{1}{g_{\beta}(s)}}\bigr)$ and $T_1=0$, we deduce from \eqref{WP3} that $u \equiv v$ on $[0,T_2]$. Therefore, in view of \eqref{WP3}, we can iterate this argument (a number of steps $\geq T/T_2$), until we extend the uniqueness result to the whole interval $[0,T]$.\\ 
\\
Following the arguments involved in the proof of Theorem \ref{globalwlowdis} together with Propositions \ref{aprop1} and \ref{aprop3}, one deduces that $ u\in C\left((0,T],H^{\infty}(\mathbb{R})\right)$ and the flow map data solution is smooth. Finally, GWP is obtained by the same reasoning in the proof of Theorem \ref{globalwlowdis}.
\end{proof}


\section{Ill-posedness result.}

In this section we prove Theorem \ref{illpos}. Let us suppose that there exists a time $T>0$ such that the Cauchy problem \eqref{fDBO} is locally well-posed in $H^s(\mathbb{R})$ on the interval $[0,T]$ and such that the flow-map data solution
$$\Phi : H^s(\mathbb{R})\longrightarrow C\left([0,T];H^s(\mathbb{R})\right), \hspace{0.5cm} u_0 \longmapsto u\left(t\right)$$  
is $C^k$ ($k=2$ and $k=3$) at the origin. Then, for each $u_0 \in H^s(\mathbb{R})$ we have that $\Phi(\cdot) u_0$ is a solution of the integral equation
$$\Phi(t)u_0=S(t)u_0-\frac{1}{2}\int_{0}^t S(t-\tau) \partial_x(\Phi(\tau) u_0)^2\, d\tau.$$
Since $\Phi(t)(0)=0$, it follows that
\begin{align*}
&u_1(t):= d_{0}\Phi(t)(u_0)=S(t) u_0,  \\
&u_2(t):=d^2_0\Phi(t)(u_0,u_0)=-\int_0^t S(t-\tau) \partial_x\left(u_1(\tau)u_1(\tau)\right)\, d\tau,  \\
&u_3(t):=d^3_0\Phi(t)(u_0,u_0,u_0)=-3\int_0^t S(t-\tau) \partial_x \left(u_1(\tau)u_2(\tau)\right)\, d\tau.  
\end{align*}
Assuming that the solution map is of class $C^k$, $k=2$ and $k=3$, we must have
\begin{equation}\label{IIP1}
\left\|u_k(t)\right\|_{H^s}\lesssim \left\|u_0 \right\|_{H^s}^k, \hspace{0.5cm} \forall u_0 \in H^s(\mathbb{R}).
\end{equation}
In the sequel we will prove that \eqref{IIP1} does not hold in general when $k=2$, assuming that $s<-\beta/2$, and when $k=3$ for $s<\min\left\{3/2-\beta,-\beta/4\right\}$. These results establish Theorem \ref{illpos}.

\subsection{\texorpdfstring{$C^2$}{Lg}-regularity.}

We divide our arguments according to parts (i) and (iii) in Theorem \ref{illpos}.

\underline{\bf Case dissipation $\beta\geq 1$}. Let $0<\alpha<\beta$ and $s<-\beta/2$ fixed. When $k=2$, we will show that \eqref{IIP1} fails for an appropriated function $u_0$. We define $u_0$ by its Fourier transform as follows
\begin{equation}\label{IIPR7}
\widehat{u_0}(\xi)=N^{-s}\omega^{-\frac{1}{2}}\left(\chi_{I_N}(\xi)+\chi_{I_N}(-\xi)\right),
\end{equation} 
where $N\gg 1$, $0<\omega \leq 1$ fixed and $I_N=[N,N+2\omega]$. A simple calculation shows that $\left\|u_0\right\|_{H^s} \sim 1$. Now, taking the Fourier transform in the space variable and changing the order of integration, it follows for all $\xi \in [-\omega/2,\omega/2]$ that
\begin{equation}\label{IIPR2}
    \begin{aligned}
    \widehat{u_2}(\xi,t) &\sim \xi e^{-i  |\xi|\xi t+(|\xi|^{\alpha}-|\xi|^{\beta})t}\int_{\mathbb{R}}\widehat{u_0}\left(\xi-\xi_1\right)\widehat{u_0}\left(\xi_1\right)\frac{e^{\sigma(\xi,\xi_1)t}-1}{\sigma(\xi,\xi_1)} \, d\xi_1  \\
&\sim N^{-2s}\omega^{-1} \xi \, e^{-i  |\xi|\xi t+(|\xi|^{\alpha}-|\xi|^{\beta})t}\, \int_{K_{\xi}}\frac{e^{\sigma(\xi,\xi_1)t}-1}{\sigma(\xi,\xi_1)} \, d\xi_1,
    \end{aligned}
\end{equation}
where
\begin{align*} 
K_{\xi}=\left\{\xi_1 \, : \, \xi-\xi_1\in I_{N}, \, \xi_1\in -I_N \right\}\cup \left\{\xi_1 \, : \,  \xi_1\in I_N, \, \xi-\xi_1\in -I_{N} \right\}
\end{align*}
and $\sigma$ is defined by
\begin{equation}\label{ipeq}
\begin{aligned}
 \sigma(\xi,\xi_1)&=i(|\xi|\xi-|\xi-\xi_1|(\xi-\xi_1)-|\xi_1|\xi_1) \\
&\hspace{0.5cm}+\left(-|\xi|^{\alpha}+|\xi|^{\beta}+|\xi-\xi_1|^{\alpha}-|\xi-\xi_1|^{\beta}+|\xi_1|^{\alpha}-|\xi_1|^{\beta}\right).
\end{aligned}
\end{equation}
Now, if $\xi \in [-\omega/2,\omega/2]$ and $\xi_1 \in K_{\xi}$, we claim
\begin{equation}\label{ipeq1}
    |\sigma(\xi,\xi_1)| \sim N^{\beta}.
\end{equation}
Indeed, simple computations show
$$|\Im{\sigma(\xi,\xi_1)}|\lesssim w N.$$
On the other hand, when $\xi \in [-\omega/2,\omega/2]$ and $\xi_1 \in K_{\xi}$ we observe
\begin{equation*}
    \begin{aligned}
     \Re{\sigma(\xi,\xi_1)}\leq |2\omega|^{\beta}+2(N+2\omega)^{\alpha}-2N^{\beta}
    \end{aligned}
\end{equation*}
and 
\begin{equation*}
    \begin{aligned}
     \Re{\sigma(\xi,\xi_1)} \geq -|2\omega|^{\alpha}+2N^{\alpha}-2(N+2\omega)^{\beta}.
    \end{aligned}
\end{equation*}
In this manner, claim \eqref{ipeq1} follows for $N$ large after collecting the above inequalities. 

Therefore, taking a fixed time $t_N=N^{-\beta-\epsilon}\in (0,T)$, $\epsilon>0$ small (but arbitrary), it follows from the Taylor expansion of the exponential function and \eqref{ipeq1} that
\begin{equation}\label{IIPR8}
\left|\frac{e^{\sigma(\xi,\xi_1)t_N}-1}{\sigma(\xi,\xi_1)}\right|= \frac{1}{N^{\beta+\epsilon}}+O\left(N^{-\beta-2\epsilon}\right).
\end{equation}
Then, since $|K_\xi|\gtrsim \omega$, \eqref{IIPR8} yields
\begin{align*}
|\widehat{u_2}(\xi,t_N)|\chi_{[-\omega/2,\omega/2]} & \gtrsim N^{-2s-\beta-\epsilon}e^{-|w/2|^{\beta}N^{-\beta-\epsilon}}|\xi|\chi_{[-\omega/2,\omega/2]}(\xi), 
\\ & \gtrsim N^{-2s-\beta-\epsilon}|\xi|\chi_{[-\omega/2,\omega/2]}(\xi), 
\end{align*}
where we used that for fixed $\omega$, taking $N$ large, $e^{-|w/2|^{\beta}N^{-\beta-\epsilon}}\gtrsim 1$. Thus, we get a lower bound for the $H^s(\mathbb{R})$-norm of $u_2(x,t_N)$,
\begin{equation}
\left\|u_2(t_N)\right\|_{H^s}^2 \gtrsim  \int_{-\omega/2}^{\omega/2} \langle\xi \rangle^{2s} |\xi|^2 N^{-4s-2\beta-2\epsilon} \, d\xi  \gtrsim N^{-4s-2\beta-2\epsilon}.
\end{equation}
The above inequality contradicts \eqref{IIP1} ($k=2$) for $N$ large enough, since $s<-\beta/2$ and $\left\|u_0\right\|_{H^s}\sim 1$. 
\ \\
\underline{ \bf Case dissipation $0<\beta< 1$}.
Let $0<\alpha<\beta$ and $s\in \mathbb{R}$. We define $u_0$ by its Fourier transform 
\begin{equation}\label{IIPR7a}
\widehat{u_0}(\xi)=\omega^{-\frac{1}{2}}  \chi_{I_1}(\xi)+ \omega^{-\frac{1}{2}}N^{-s}\chi_{I_2}(\xi),
\end{equation} 
with $I_1=[\omega/2, \omega ]$, $I_2=[N, N+\omega]$ and $N\gg 1$, $\omega \ll N$ to be chosen later. Then $\left\|u_0\right\|_{H^s}\sim 1$. Computing the Fourier transform of $u_2(t)$ leads to
\begin{equation*}
    \widehat{u_2}(\xi,t) \sim \xi e^{-i  |\xi|\xi t+(|\xi|^{\alpha}-|\xi|^{\beta})t}\int_{\mathbb{R}}\widehat{u_0}\left(\xi-\xi_1\right)\widehat{u_0}\left(\xi_1\right)\frac{e^{\sigma(\xi,\xi_1)t}-1}{\sigma(\xi,\xi_1)} \, d\xi_1 , 
\end{equation*}
where $\sigma(\xi,\xi_1)$ was defined in \eqref{ipeq}. By support considerations, we have $\left\|u_2(t)\right\|_{H^s} \geq \left\| v_2(t)\right\|_{H^s}$ where
\begin{equation*}
    \widehat{v_2}(\xi,t) \sim N^{-s}\omega^{-1}\xi e^{-i  |\xi|\xi t+(|\xi|^{\alpha}-|\xi|^{\beta})t}\int_{K_{\xi}} \frac{e^{\sigma(\xi,\xi_1)t}-1}{\sigma(\xi,\xi_1)} \, d\xi_1  
\end{equation*}
and 
\begin{equation*}
K_{\xi}=\left\{\xi_1 \, : \, \xi_1\in I_{1}, \, \xi-\xi_1\in I_2 \right\}\cup \left\{\xi_1 \, : \,  \xi_1\in I_2, \, \xi-\xi_1\in I_{1} \right\}.
\end{equation*}
We see that if $\xi_1 \in K_{\xi}$, then $\xi \in [N+\omega/2, N+2\omega]$ and 
\begin{equation*}
    \begin{aligned}
    | \Re{\sigma(\xi,\xi_1)}|&\lesssim N^{\beta}     \\
     \Im{\sigma(\xi,\xi_1)}&=2\xi_1(\xi_1-\xi) \sim \omega N.
    \end{aligned}
\end{equation*}
We deduce for $\omega =N^{\beta-1}\ll N$ that $|\sigma (\xi, \xi_1)|\sim N^{\beta}$. Consider $t_N=(N+2\beta)^{-\beta -\epsilon}\sim N^{-\beta-\epsilon}$. By a Taylor expansion of the exponential function, 
\begin{equation*}
\begin{aligned}
\frac{e^{\sigma(\xi,\xi_1)t}-1}{\sigma(\xi,\xi_1)}=t_N+R(t_N, \xi, \xi_1) 
\end{aligned}
\end{equation*}
and
\begin{equation*}
\begin{aligned}
|R(t_N, \xi, \xi_1)|\lesssim \sum\limits_{k\geq 2}\frac{t_N^k |\sigma(\xi,\xi_1)|^{k-1}}{k!}\lesssim N^{-\beta-2\epsilon}.
\end{aligned}
\end{equation*}
Since $|K_\xi|\sim \omega$, we have that
\begin{align*}
|\widehat{v_2(t_N)}(\xi)|&\gtrsim N^{-s+1}\omega^{-1}e^{-(N+2\omega)^{-\epsilon}}\omega N^{-\beta-\epsilon}\chi_{[N+\omega/2, N+2\omega]}(\xi) \\
&\gtrsim N^{-s+1-\beta -\epsilon}\chi_{[N+\omega/2, N+2\omega]}(\xi).
\end{align*}
Then, the lower bound for the $H^s$-norm of $u_2(x, t_N)$
\begin{equation*}
\left\|u_2(t_N)\right\|_{H^s}\gtrsim N^{-s+1-\beta -\epsilon}\Bigl(\int_{N+\omega/2}^{N+2\omega} (1+|\xi|^2)^s\,d\xi\Bigr)^{1/2}\sim N^{1-\beta-\epsilon}\omega^{1/2}\sim N^{(1-\beta)/2 -\epsilon}.
\end{equation*}
This inequality contradicts \eqref{IIP1} ($k=2$) for $N$ large enough, $\epsilon \ll 1$ and $\beta<1$.

\subsection{\texorpdfstring{$C^3$}{Lg}-regularity.}

Suppose that $s<\min\left\{3/2-\beta,-\beta/4\right\}$. When $k=3$, we will show that \eqref{IIP1} fails for an appropriated function $u_0$. We consider $u_0$ as in \eqref{IIPR7}, i.e.,
\begin{equation*}
\widehat{u_0}(\xi)=N^{-s}\omega^{-1/2}\left(\chi_{I_N}(\xi)+\chi_{I_N}(-\xi)\right),
\end{equation*}
$N\gg 1$, $I_N=[N,N+2\omega]$ and here $\omega \ll N$ to be chosen later. Observe that $\left\|u_0\right\|_{H^s}\sim 1$.
Computing the Fourier transform of $u_3(t)$ for each $\xi \in \mathbb{R}$ one gets
\begin{equation}\label{IIPR1}
\widehat{u_3(t)}(\xi)\sim \xi \, \int_{0}^t e^{-i|\xi|\xi (t-\tau)+(|\xi|^{\alpha}-|\xi|^{\beta})(t-\tau)} \widehat{u_1(\tau)} \ast \widehat{u_2(\tau)}(\xi) \,  d\tau.
\end{equation}
Therefore, in view of \eqref{IIPR2} and Fubini's Theorem one finds
\begin{equation}\label{IIPR5}
\begin{aligned}
\widehat{u_3(t)}(\xi)\sim & \, \xi  e^{-i  |\xi|\xi t+(|\xi|^{\alpha}-|\xi|^{\beta})t} \int_{\mathbb{R}^2} \widehat{u_0}(\xi_1)\widehat{u_0}(\xi_2-\xi_1)\widehat{u_0}(\xi-\xi_2)\, \frac{ \xi_2}{\sigma(\xi_2,\xi_1)}  \\
& \hspace{1cm} \times \left(\frac{e^{\eta(\xi,\xi_1,\xi_2)t}-1}{\eta(\xi,\xi_1,\xi_2)} -\frac{e^{\sigma(\xi,\xi_2)t}-1}{\sigma(\xi,\xi_2)}\right) \, d\xi_1 \, d\xi_2,
\end{aligned}
\end{equation}
where $\sigma(\xi,\xi_1)$ is given by \eqref{ipeq1} and we have set
$$\eta(\xi,\xi_1,\xi_2):=\sigma(\xi,\xi_2)+\sigma(\xi_2,\xi_1).$$
By support considerations,
\begin{equation}\label{IIPR3}
|\widehat{u_3(t)}(\xi) | \gtrsim  N^{-3s}\omega^{-3/2}\left|\, \xi e^{(|\xi|^{\alpha}-|\xi|^{\beta})t} \int_{K_{\xi}} \frac{ \xi_2}{\sigma(\xi_2,\xi_1)} \left(\frac{e^{\eta(\xi,\xi_1,\xi_2)t}-1}{\eta(\xi,\xi_1,\xi_2)} -\frac{e^{\sigma(\xi,\xi_2)t}-1}{\sigma(\xi,\xi_2)}\right) \, d\xi_1 \, d\xi_2\right|,
\end{equation}
where $K_{\xi}=K_{\xi}^1\cup K_{\xi}^2 \cup K_{\xi}^3$ and
\begin{align*}
K_{\xi}^1=\left\{(\xi_1,\xi_2):\xi_1\in I_N, \, \xi_2-\xi_1\in I_N, \, \xi-\xi_2\in -I_N \right\}, \\
K_{\xi}^2=\left\{(\xi_1,\xi_2):\xi_1\in I_N, \, \xi_2-\xi_1\in -I_N, \, \xi-\xi_2\in I_N \right\}, \\
K_{\xi}^3=\left\{(\xi_1,\xi_2):\xi_1\in -I_N, \, \xi_2-\xi_1\in I_N, \, \xi-\xi_2\in I_N \right\}. 
\end{align*}
We will restrict the values of $\xi$ to the interval $ [N+3\omega,N+4\omega]$. Then, under this condition and using that $\omega \ll N$, it follows for $(\xi_1,\xi_2)\in K_{\xi}$ that
\begin{equation}\label{phasesti}
    \begin{aligned}
    \Im{\eta(\xi,\xi_1,\xi_2)} \sim \omega^2,\,\text{ and }\, |\Re{\eta(\xi,\xi_1,\xi_2)}|\sim N^{\beta}.
    \end{aligned}
\end{equation}
We will divide our arguments in two cases subcases depending on the dissipation parameter $\beta$.

\underline{ \bf Case dissipation $1\\
\leq \beta < 2$}.  In view of \eqref{phasesti}, we are led to choose $\omega=N^{\beta/2} \ll N$ with $N\gg 1$. Hence $|\eta(\xi,\xi_1,\xi_2)|\sim N^{\beta}$ and since $1\leq \beta <2$, 
\begin{equation}\label{phaseesti1}
    \left|\frac{\xi_2}{\sigma(\xi_2,\xi_1)}\right| \sim N^{-1}.
\end{equation}
Next, we consider
\begin{equation}\label{Comtrtime}
    t_N:=N^{-\beta-\epsilon}, 
\end{equation}
with $0<\epsilon \ll 1$. We divide the estimate of $|\widehat{u}_3(t_N)|(\xi)$ on $[N+3\omega,N+4\omega]$ as follows
\begin{equation*}
    \begin{aligned}
     |\widehat{u_3(t_N)}(\xi) |\chi_{[N+3\omega,N+4\omega]}(\xi) &\gtrsim  N^{-3s+1}\omega^{-3/2}\left| \int_{K_{\xi}} \frac{ \xi_2}{\sigma(\xi_2,\xi_1)} \left(\frac{e^{\eta(\xi,\xi_1,\xi_2)t_N}-1}{\eta(\xi,\xi_1,\xi_2)}\right) \, d\xi_1 \, d\xi_2\right|\chi_{[N+3\omega,N+4\omega]}(\xi) \\
     &\hspace{0.2cm}-N^{-3s+1}\omega^{-3/2}\left| \int_{K_{\xi}} \frac{ \xi_2}{\sigma(\xi_2,\xi_1)} \left(\frac{e^{\sigma(\xi,\xi_2)t_N}-1}{\sigma(\xi,\xi_2)}\right) \, d\xi_1 \, d\xi_2\right|\chi_{[N+3\omega,N+4\omega]}(\xi) \\
     &=B_1-B_2.
    \end{aligned}
\end{equation*}
To estimate $B_1$, we observe
\begin{equation*}
   \frac{e^{\eta(\xi,\xi_1,\xi_2)t_N}-1}{\eta(\xi,\xi_1,\xi_2)}=t_N+O(N^{-\beta-2\epsilon}).
\end{equation*}
This yields in view of \eqref{phaseesti1} and $|K_{\xi}|\sim \omega^2$ to
\begin{equation}
    \begin{aligned}
     B_1 &\gtrsim N^{-3s+1}\omega^{-3/2}N^{-1}\omega^2N^{-\beta-\epsilon}\chi_{[N+3\omega,N+4\omega]}(\xi) \\
     &\sim N^{-3s-3\beta/4-\epsilon}\chi_{[N+3\omega,N+4\omega]}(\xi).
    \end{aligned}
\end{equation}
To deal with $B_2$, we observe $|\sigma(\xi,\xi_2)|\gtrsim \omega N$ for $(\xi,\xi_2)\in K_{\xi}$ and $\xi\in[N+3\omega,N+4\omega]$. Thus, since $|\sigma(\xi,\xi_2)t_N| \lesssim 1$, we get
\begin{equation}
    \begin{aligned}
     B_2 &\lesssim N^{-3s+1}\omega^{-3/2}N^{-2}\omega \chi_{[N+3\omega,N+4\omega]}(\xi) \\
     &\sim N^{-3s-1-\beta/4}\chi_{[N+3\omega,N+4\omega]}(\xi).
    \end{aligned}
\end{equation}
Since $-3s-\beta/4-1< -3s-3\beta/4-\epsilon$ given that $\beta<2$, we conclude 
$$B_1-B_2 \gtrsim N^{-3s-3\beta/4-\epsilon}\chi_{[N+3\omega,N+4\omega]}(\xi),$$
for $N \gg 1$. Thus, from this fact we derive the following lower bound for the $H^s$-norm of $u_3(x,t_N)$, 
\begin{equation*}
    \left\|u_3(t_N)\right\|_{H^s}\gtrsim N^{-3s-3\beta/4-\epsilon}\omega^{1/2}N^s\sim N^{-2s-\beta/2-\epsilon}.
\end{equation*}
The above inequality contradicts \eqref{IIP1} ($k=3$) for $N$ large given that $\left\|u_0\right\|_{H^s}\sim 1$, $s<-\beta/4$ and $0<\epsilon \ll 1$.

\underline{ \bf Case dissipation $\beta \geq 2$}. In this case the contributions of $B_1$ and $B_2$ are equivalent and so we require of a different estimate to bound \eqref{IIPR3}. Let $\omega=\epsilon_1 N$ with $0<\epsilon_1 \ll 1$ to be chosen later. We first observe that \eqref{phasesti} shows that $\eta(\xi,\xi_1,\xi_2)\sim N^{\beta}$. Moreover, we claim
\begin{align}
\left|\sigma \left(\xi,\xi_2\right)\right|\sim N^{\beta}, \label{aeq10} \\
\left|\sigma\left(\xi_2,\xi_1\right)\right|\sim N^{\beta}, \label{aeq11} 
\end{align}
for $\epsilon_1>0$ small enough, $N$ sufficiently large and $(\xi_1,\xi_2)\in K_{\xi}$ with $\xi\in [N+3\omega,N+4\omega]$. For the sake of brevity, we will only give a proof to \eqref{aeq10}, since \eqref{aeq11} follows in a similar manner.  On these terms, since the imaginary part of $\sigma(\xi,\xi_2)$ is of order $O(N^2)$, we are reduced to show that 
$|\Re{\sigma(\xi,\xi_1)}|\sim N^{\beta}$.

Suppose that $(\xi_1,\xi_2)\in K_{\xi}^1$. So, since $2N \leq \xi_2 \leq 2N+4\omega$, we deduce
\begin{equation*}
    \begin{aligned}
     \Re{\sigma(\xi,\xi_2)} \leq  N^{\alpha}((2+4\epsilon_1)^{\alpha}+(1+2\epsilon_1)^{\alpha}-(1+3\epsilon_1)^{\alpha})-N^{\beta}((2+3\epsilon_1)^{\beta}+1-(1+4\epsilon_1)^{\beta}),
    \end{aligned}
\end{equation*}
and
\begin{equation*}
    \begin{aligned}
     \Re{\sigma(\xi,\xi_2)} \geq   N^{\alpha}((2+3\epsilon_1)^{\alpha}+1-(1+4\epsilon_1)^{\alpha})-N^{\beta}((2+4\epsilon_1)^{\beta}+(1+2\epsilon_1)^{\beta}-(1+3\epsilon_1)^{\beta}),
    \end{aligned}
\end{equation*}
which clearly leads to $|\Re{\sigma(\xi,\xi_1)}|\sim N^{\beta}$ for $\epsilon_1$ small enough.

On the other hand, when $(\xi_1,\xi_2)\in K_{\xi}^2\cup K_{\xi}^3$, we have $\omega \leq \xi_2 \leq 2\omega$ and it follows 

\begin{equation*}
    \begin{aligned}
     \Re{\sigma(\xi,\xi_2)} \leq -N^{\alpha}((1+3\epsilon_1)^{\alpha}-(1+2\epsilon_1)^{\alpha}-(2\epsilon_1)^{\alpha})+N^{\beta}((1+4\epsilon_1)^{\beta}-1-\epsilon_1^{\beta})
    \end{aligned}
\end{equation*}
and
\begin{equation*}
    \begin{aligned}
     \Re{\sigma(\xi,\xi_2)} \geq   -N^{\alpha}((1+4\epsilon_1)^{\alpha}-1-\epsilon_1^{\alpha})+N^{\beta}((1+3\epsilon_1)^{\beta}-(1+2\epsilon_1)^{\beta}-(2\epsilon_1)^{\beta}).
    \end{aligned}
\end{equation*}
Note that in this case the constants with factor $N^{\beta}$ tend to zero as $\epsilon_1 \to 0$ and they are always positive if $\epsilon_1>0$. To see this, the mean value inequality yields
\begin{equation}
    \begin{aligned}
   (1+4\epsilon_1)^{\beta}-1-\epsilon_1^{\beta} &\geq (1+3\epsilon_1)^{\beta}-(1+2\epsilon_1)^{\beta}-(2\epsilon_1)^{\beta} \\
   &\geq \beta \epsilon_1 \left((1+2\epsilon_1)^{\beta-1}-\frac{2}{\beta}(2\epsilon_1)^{\beta-1}\right) >0,
    \end{aligned}
\end{equation}
given that $\beta \geq 2$ and $\epsilon_1>0$.

Therefore, gathering the above estimates, we can choose $\epsilon_1>0$ small to fix the sign of the constants involving $N^{\beta}$. Consequently, we take $N$ large to absorb the terms with $N^{\alpha}$. At the end, we will find that $|\Re{\sigma(\xi,\xi_2 )}|\sim N^{\beta}$ as claimed.

Next, we consider $t_N$ as in \eqref{Comtrtime} with $0<\epsilon \ll 1$. By the Taylor expansion of the exponential function we find
\begin{equation}\label{IIPR4}
\frac{1}{\sigma(\xi_2,\xi_1)}\left(\frac{e^{\eta(\xi,\xi_1,\xi_2)t_N}-1}{\eta(\xi,\xi_1,\xi_2)} -\frac{e^{\sigma(\xi,\xi_2)t_N}-1}{\sigma(\xi,\xi_2)}\right)=\frac{1}{N^{2\beta+2\epsilon}}+R(\xi,\xi_1,\xi_z),
\end{equation}
where
$$\left|R(\xi,\xi_1,\xi_z)\right|\leq \sum_{k=3}^{\infty}\left| \frac{\eta(\xi,\xi_1,\xi_2)^{k-1}-\sigma(\xi,\xi_2)^{k-1}}{\sigma(\xi_2,\xi_1)\, k!}\right|t_{N}^k\leq O\left(\frac{1}{N^{2\beta+3\epsilon}}\right).$$

In view of the above inequality the main contribution of \eqref{IIPR4} is given by $N^{-2\beta-2\epsilon}$. Thus, for $N$ large,
\begin{equation}\label{IIPR6}
\Re\left(\frac{1}{\sigma(\xi_2,\xi_1)}\left(\frac{e^{\eta(\xi,\xi_1,\xi_2)t_N}-1}{\eta(\xi,\xi_1,\xi_2)} -\frac{e^{\sigma(\xi,\xi_2)t_N}-1}{\sigma(\xi,\xi_2)}\right)\right) \gtrsim  N^{-2\beta-2\epsilon}.
\end{equation} 
Since $\xi_2 \sim N$ and $|K_{\xi}|\sim \omega^2$, we get from \eqref{IIPR6} that
\begin{align*}
&\left|\widehat{u_3(t_{N})}(\xi)\right| \\
&\gtrsim N^{-3s}\omega^{-3/2}|\xi| e^{(|\xi|^{\alpha}-|\xi|^{\beta})t_{N}} \left|\int_{K_{\xi}} \frac{ \xi_2}{\sigma(\xi_2,\xi_1)} \left(\frac{e^{\eta(\xi,\xi_1,\xi_2)t}-1}{\eta(\xi,\xi_1,\xi_2)} -\frac{e^{\sigma(\xi,\xi_2)t}-1}{\sigma(\xi,\xi_2)}\right) \, d\xi_1 \, d\xi_2 \right| \\
& \gtrsim N^{-3s}\omega^{-3/2}N \left|\Re\left( \int_{K_{\xi}} \frac{ \xi_2}{\sigma(\xi_2,\xi_1)} \left(\frac{e^{\eta(\xi,\xi_1,\xi_2)t}-1}{\eta(\xi,\xi_1,\xi_2)} -\frac{e^{\sigma(\xi,\xi_2)t}-1}{\sigma(\xi,\xi_2)}\right) \, d\xi_1 \, d\xi_2\right) \right| \\
& \gtrsim N^{-3s+2}\omega^{-3/2}N^{-2\beta-2\epsilon} \omega^2. 
\end{align*}
The above inequality gives the following lower bound for the $H^s$-norm of $u_3$,
\begin{align*}
\left\|u_{3}(t_N)\right\|_{H^s}^2 & \geq \int_{\mathbb{R}}\langle\xi \rangle^{2s} \left|\widehat{u_3(t_N)}(\xi)\right|^2 \chi_{[N+3\omega,N+4\omega]}(\xi)\, d\xi \\
&\gtrsim \omega^2 N^{-4s+4-4\beta-4\epsilon} \\
&\sim N^{-4s+6-4\beta-4\epsilon},
\end{align*}
which in turn contradicts \eqref{IIP1} for $N$ large given that $\left\|u_0\right\|_{H^s}\sim 1$, $s<3/2-\beta$ and $\epsilon>0$ is arbitrary.


\section{Proof of Proposition \ref{contInsta}}

This section is aimed to establish Proposition \ref{contInsta} in which we consider the limit behavior of the family \eqref{fDBO} when $\alpha \to \beta^{-}$. We first introduce some notation and preliminaries. Since our arguments are based on energy estimates, we require the following commutator relation deduced by Kato and Ponce in \cite{KP}.
\begin{lemma}\label{conmKP}
If $s>0$ and $1<p<\infty$, then
\begin{equation}
    \left\|[J^s,f]g\right\|_{L^p(\mathbb{R})} \lesssim  \left\|\partial_x f\right\|_{L^{\infty}(\mathbb{R})} \left\|J^{s-1}g\right\|_{L^p(\mathbb{R})}+ \left\|J^s f\right\|_{L^p(\mathbb{R})} \left\|g\right\|_{L^{\infty}(\mathbb{R})},
\end{equation}
where
$$[J^s,f]g=J^s(fg)-fJ^sg.$$
\end{lemma}
As it was stated in Proposition \ref{contInsta}, for fixed $\beta>1$ and $u_0\in H^{s_0}(\mathbb{R})$, we denote by $u^{\alpha}$ the solutions of the IVP \eqref{fDBO} with growth order $0<\alpha\leq\beta$, dissipation order $\beta$ and initial datum $u_{\alpha}(0)=u_0$, that is,
\begin{equation}\label{fDBOPA}
\left\{
\begin{aligned}
&u_t^{\alpha}+\mathcal{H}u_{xx}^{\alpha}-(D_x^{\alpha}-D_x^{\beta})u^{\alpha}+u^{\alpha}u_x^{\alpha}=0, \qquad && x\in \mathbb{R}  \, , \quad t > 0,  \\
&u^{\alpha}(x,0)=u_0(x), 
\end{aligned}
\right.
\end{equation}
In particular $u^{\beta}$ denotes the solution of the Benjamin-Ono equation. Regarding existence of solutions, since $s_0>3/2$, when $0<\alpha<\beta$, the results in Theorem \ref{globalwlowdis} if $1<\beta<2$ or in Theorem \ref{globalw} if $\beta\geq 2$ establish that there exists 
\begin{equation}\label{ExistSOl}
    u^{\alpha}\in C([0,\infty);H^{s_0}(\mathbb{R}))
\end{equation}
solution of \eqref{fDBOPA}. This same conclusion holds when $\alpha=\beta$ due to the GWP theory established in  \cite{Iorio,Ponce1991}. For the sake of brevity, in what follows we may assume that $u^{\beta}$ is sufficiently regular to perform all of the subsequent estimates. Indeed, this remark can be justified approximating $u_0$ by smooth functions, using the continuous dependence of the solution flow associated to BO and taking the limit in our arguments. Notice that in contrast Theorem \ref{globalwlowdis} and \ref{globalw} assure that $u^{\alpha}(t)$ is smooth whenever $0<\alpha<\beta$ and $t>0$.  
\\ \\
We are in condition to prove the first part of Proposition \ref{contInsta} in which we establish some point-wise bound for solutions of \eqref{fDBOPA}.

\begin{lemma}\label{existT}
Let $\beta>1$, $u_0\in H^{s_0}(\mathbb{R})$ with $s_0>3/2$. For each $\alpha\in(0,\beta]$, let $u^{\alpha}\in  C([0,\infty);H^{s_0}(\mathbb{R}))$  be the corresponding solutions of \eqref{fDBOPA} with initial data $u_0$. Then there exist $T>0$ and a function $g\in C([0,T];[0,\infty))$ such that
\begin{equation*}
\left\|u^{\alpha}(t)\right\|_{H^{s_0}}\leq g(t), \hspace{0.5cm} t\in [0,T].
\end{equation*}
\end{lemma}
\begin{proof}
Applying the operator $J^{s_0}$ to \eqref{fDBOPA} and multiplying the resulting expression by $J^{s_0}u^{\alpha}$ yields 
\begin{equation}\label{wpeq1}
  \frac{d}{dt}\left\|u^{\alpha}(t)\right\|_{H^{s_0}}^2=\int (D_x^{\alpha}-D_x^{\beta})J^{s_0}u^{\alpha}(t) J^{s_0} u^{\alpha}(t) \, dx -\int J^{s_0}(u^{\alpha}(t)\partial_xu^{\alpha}(t))J^{s_0}u^{\alpha}(t) \, dx
\end{equation}
In view of Plancherel's identity
\begin{equation}\label{wleq0}
\begin{aligned}
    \int (D_x^{\alpha}-D_x^{\beta})J^{s_0}u^{\alpha}(t) J^{s_0} u^{\alpha}(t)\, dx &\leq \int_{|\xi|\leq 1} (|\xi|^{\alpha}-|\xi|^{\beta})|\langle \xi \rangle^{s_0}\widehat{u^{\alpha}}(\xi,t)|^{2} \, d\xi \leq 2 \left\|u^{\alpha}(t)\right\|_{H^{s_0}}^2.
  \end{aligned}  
\end{equation}
From the above expression, Lemma \ref{conmKP} to deal with second term on the right-hand side of \eqref{wpeq1} and Sobolev's embedding $\left\|u^{\alpha}_x\right\|_{L^\infty}\lesssim \left\|u^{\alpha}\right\|_{H^{s_0}}$, there exists a constant $c>0$ independent of $\alpha$ such that
\begin{equation}\label{wleq1}
    \frac{d}{dt}\left\|u^{\alpha}(t)\right\|_{H^{s_0}}^2 \leq c \big( \left\|u^{\alpha}(t)\right\|_{H^{s_0}}^2+\left\|u^{\alpha}(t)\right\|_{H^{s_0}}^3\big).
\end{equation}
Letting $h(t)=\left\|u^{\alpha}(t)\right\|_{H^{s_0}}^2e^{-ct}$, \eqref{wleq1} yields the differential inequality
\begin{equation}
    \begin{aligned}
    -\frac{d}{dt}\left(\frac{1}{h(t)^{1/2}}\right)\leq \frac{c}{2}e^{ct/2}, \hspace{0.5cm} h(0)=\left\|u_0\right\|_{H^{s_0}}^2.
    \end{aligned}
\end{equation}
Integrating the above expression between $0$ and $t$ we find 
\begin{equation}
    \frac{1}{h(0)^{1/2}}-\frac{1}{h(t)^{1/2}} \leq \big(e^{ct/2}-1\big)
\end{equation}
and so replacing $h(t)$ by $\left\|u^{\alpha}(t)\right\|_{H^{s_0}}^2e^{-ct}$ and solving for $\left\|u^{\alpha}(t)\right\|_{H^{s_0}}$ we arrive at
\begin{equation}
    \left\|u^{\alpha}(t)\right\|_{H^{s_0}} \leq \frac{\left\|u_0\right\|_{H^{s_0}}e^{ct/2}}{1-\left\|u_0\right\|_{H^{s_0}}(e^{ct/2}-1)}=:g(t).
\end{equation}
Consequently the above inequality concludes the proof of the lemma after taking any fixed time $T$ in the interval $(0,\frac{2}{c}\log\big(\frac{1+\left\|u_0\right\|_{H^{s_0}}}{\left\|u_0\right\|_{H^{s_0}}}\big)\big)$.
\end{proof}

Next we prove \eqref{ContMap}. Here we assume that $s_0>3/2+\max\left\{\beta/2,1\right\}$ and $s<s_0-\max\left\{\beta/2,1\right\}$. By continuity of the Sobolev embedding, it is sufficient to establish \eqref{ContMap} when $3/2<s<s_0-\max\left\{\beta/2,1\right\}$. In this manner, by Lemma \ref{existT}, there exist $T>0$, $M>0$ such that
\begin{equation}\label{boundUNIF}
    \sup_{0<\alpha\leq \beta}\sup_{t\in[0,T]}\left\|u^{\alpha}(t)\right\|_{H^{s_0}}  \leq M,
\end{equation}
where $u^{\alpha}\in C([0,T];H^{s_0}(\mathbb{R}))$. Let us first establish continuity at the left-hand side of $\beta$. Let $0<\alpha<\beta$ and define $w=w(\alpha,\beta)=u^{\alpha}-u^{\beta}$, we find that $w$ solves
\begin{equation}\label{enerE1}
\left\{
    \begin{aligned}
    &w_t+\mathcal{H}w_{xx}-(D_x^{\alpha}-D_x^{\beta})w-(D_x^{\alpha}-D_x^{\beta})u^{\beta}+wu^{\alpha}_x+u^{\beta}w_x=0, \hspace{0.5cm} x\in \mathbb{R}, \, 0<t<T, \\
    &w(0)=0.
    \end{aligned} \right.
\end{equation}
After applying $J^{s}$ to the equation in \eqref{enerE1} and multiplying the resulting expression by $J^sw$ it is deduced 
\begin{equation}\label{diffineq}
\begin{aligned}
 \frac{1}{2}\frac{d}{dt}\left\|w(t)\right\|_{H^{s}}^2=&\int (D_x^{\alpha}-D_x^{\beta})J^sw J^sw \, dx +\int (D_x^{\alpha}-D_x^{\beta})J^su^{\beta} J^sw\, dx \\
 &-\int J^s(wu^{\alpha}_x)J^sw\, dx-\int J^s(u^{\beta} w_x)J^sw\, dx \\
 =: &\mathcal{I}_{\alpha}+\mathcal{II}_{\alpha}+\mathcal{III}_{\alpha}+\mathcal{IV}_{\alpha}.
\end{aligned}
\end{equation}
Let us estimate the contribution of each term on the right-hand side of the above equation. By Plancherel's identity and \eqref{boundUNIF}, we find
\begin{equation}\label{EstimateI}
\begin{aligned}
    \mathcal{I}_{\alpha}\leq \int_{|\xi|\leq 1} (|\xi|^{\alpha}-|\xi|^{\beta}\big)|\widehat{J^sw}(\xi,t)|^2\, d\xi &\leq \max_{0<|\xi|\leq 1}\big\{|\xi|^{\alpha}-|\xi|^{\beta}\big\}\left\|w(t)\right\|_{H^{s_0}}^2 \\
    &\lesssim\big( \frac{\alpha}{\beta}\big)^{\alpha/(\beta-\alpha)}\left(\frac{\beta-\alpha}{\beta}\right)M^2,
\end{aligned}
\end{equation}
where we have used that the function $f(x)=x^{\alpha}-x^{\beta}$ on $[0,1]$ reaches its maximum at $x_{max}=\left(\alpha/\beta\right)^{1/(\beta-\alpha)}$. Using that $|\xi|^{\alpha/2}+|\xi|^{\beta/2}\lesssim \langle \xi \rangle^{\beta/2}$ and that $s_0>s+\beta/2$, 
\begin{equation}\label{EstimateII}
\begin{aligned}
 |\mathcal{II}_{\alpha} |&=\big|\int (|\xi|^{\alpha/2}-|\xi|^{\beta/2})\widehat{J^s u^{\beta}}\overline{(|\xi|^{\alpha/2}+|\xi|^{\beta/2})\widehat{J^s w}} \, dx \big| \\
 &\lesssim \left\|(|\xi|^{\alpha/2}-|\xi|^{\beta/2})\widehat{J^su^{\beta}}\right\|_{L^2}\left\|w\right\|_{H^{s+\beta/2}} \lesssim M\left\|(|\xi|^{\alpha/2}-|\xi|^{\beta/2})\widehat{J^su^{\beta}}\right\|_{L^2}.
\end{aligned}
\end{equation}
Writing
\begin{equation*}
    \begin{aligned}
    \mathcal{III}_{\alpha}=-\int [J^s,w]u^{\alpha}_x J^s w\, dx-\int wJ^su_x^{\alpha} J^sw \, dx,
    \end{aligned}
\end{equation*}
we can use the above display, Lemma \ref{conmKP}, Sobolev embedding  $\left\|w_x\right\|_{L^{\infty}} \lesssim \left\|w\right\|_{H^{3/2^{+}}}$, the fact that $s_0>s+1$, $s>3/2$ and \eqref{boundUNIF} to infer 
\begin{equation}\label{EstimateIII}
    \begin{aligned}
    |\mathcal{III}_{\alpha}| &\lesssim \left\|w_x\right\|_{L^{\infty}}\left\|u^{\alpha}\right\|_{H^{s}}\left\|w\right\|_{H^{s}}+\left\|u^{\alpha}_x\right\|_{L^{\infty}}\left\|w\right\|_{H^s}^2+\left\|w\right\|_{L^{\infty}}\left\|u^{\alpha}\right\|_{H^{s+1}}\left\|w\right\|_{H^s}\\
    &\lesssim M\left\|w\right\|_{H^{s}}^2.
    \end{aligned}
\end{equation}
Integration by parts shows
\begin{equation*}
    \begin{aligned}
    \mathcal{IV}_{\alpha}&=-\int [J^s,u^{\beta}]w_x J^s w\, dx-\int u^{\beta}J^sw_x J^sw \, dx =-\int [J^s,u^{\beta}]w_x J^s w\, dx+\frac{1}{2}\int u_x^{\beta} \,|J^sw|^2 \, dx,
    \end{aligned}
\end{equation*}
which together with Lemma \ref{conmKP} and \eqref{boundUNIF} imply
\begin{equation}\label{EstimateIV}
    |\mathcal{IV}_{\alpha}| \lesssim \left\|u_x^{\beta}\right\|_{L^{\infty}}\left\|w\right\|_{H^s}^2+\left\|u^{\beta}\right\|_{H^s}\left\| w_x\right\|_{L^{\infty}}\left\|w\right\|_{H^s} \lesssim M\left\|w\right\|_{H^s}^2.
\end{equation}
Plugging \eqref{EstimateI}, \eqref{EstimateII}, \eqref{EstimateIII} and \eqref{EstimateIV} on the right-hand side of \eqref{diffineq}, there exists a constant $c=c(M)$ depending on $M$ as in \eqref{boundUNIF} and independent of $\alpha$ such that
\begin{equation*}
    \frac{d}{dt}\left\|w(t)\right\|_{H^s}^2\leq c\big( \frac{\alpha}{\beta}\big)^{\alpha/(\beta-\alpha)}\left(\frac{\beta-\alpha}{\beta}\right)+c\left\|(|\xi|^{\alpha/2}-|\xi|^{\beta/2})\widehat{J^su^{\beta}}(t)\right\|_{L^2}+c\left\|w(t)\right\|_{L^2}^2.
\end{equation*}
or equivalently
\begin{equation*}
    \frac{d}{dt}\big(\left\|w(t)\right\|_{H^s}^2e^{-ct}\big)\leq c\big( \frac{\alpha}{\beta}\big)^{\alpha/(\beta-\alpha)}\left(\frac{\beta-\alpha}{\beta}\right)e^{-ct}+c\left\|(|\xi|^{\alpha/2}-|\xi|^{\beta/2})\widehat{J^su^{\beta}}(t)\right\|_{L^2}e^{-ct}.
\end{equation*}
Integrating the last inequality between $0$ and $t$ and recalling that $w(0)=u^{\alpha}(0)-u^{\beta}(0)=0$, we get
\begin{equation}\label{approxiconcl}
    \left\|u^{\alpha}(t)-u^{\beta}(t)\right\|_{H^{s}}^2\leq  \big( \frac{\alpha}{\beta}\big)^{\alpha/(\beta-\alpha)}\left(\frac{\beta-\alpha}{\beta}\right)e^{cT}+ce^{cT}\int_0^T\left\|(|\xi|^{\alpha/2}-|\xi|^{\beta/2})\widehat{J^su^{\beta}}(\tau)\right\|_{L^2}\, d \tau.
\end{equation}
Therefore, to obtain continuity as $\alpha \to \beta^{-}$, the right-hand side of \eqref{approxiconcl}  reduces our considerations to establish
\begin{equation}\label{limitcont}
    \int_0^T\left\|(|\xi|^{\alpha/2}-|\xi|^{\beta/2})\widehat{J^su^{\beta}}(\tau)\right\|_{L^2}\, d \tau \to 0 \hspace{0.5cm} \text{as } \alpha\to \beta.
\end{equation}
Using again that $||\xi|^{\alpha/2}-|\xi|^{\beta/2}|\lesssim \langle \xi \rangle^{\beta/2}$,
\begin{equation}\label{limitcont1}
    \left|(|\xi|^{\alpha}-|\xi|^{\beta})\widehat{J^su^{\beta}}(\xi,\tau)\right|^2 \lesssim
 \left|\widehat{J^{s+\beta/2}u^{\beta}}(\xi,\tau)\right|^2
 \end{equation}
 and so given that $s_0>s+\beta/2$,
 \begin{equation}\label{limitcont2}
     \left\|(|\xi|^{\alpha/2}-|\xi|^{\beta/2})\widehat{J^su^{\beta}}(\tau)\right\|_{L^2} \lesssim \sup_{t\in[0,T]}\left\|u^{\beta}(t)\right\|_{H^{s+\beta/2}} \lesssim M.
 \end{equation}
In this manner, since point-wise $|\xi|^{\alpha}-|\xi|^{\beta}\to 0$ as $\alpha\to \beta^{-}$, \eqref{limitcont1} and Lebesgue dominated convergence theorem yield $\left\|(|\xi|^{\alpha/2}-|\xi|^{\beta/2})\widehat{J^su}(\tau)\right\|_{L^2} \to 0$, $\tau \in (0,T)$. This conclusion, \eqref{limitcont2} and Lebesgue dominated convergence theorem (on the time variable $\tau$) imply \eqref{limitcont}. From this, we get continuity as $\alpha \to \beta^{-}$ for the function $\alpha \in (0,\beta) \mapsto u^{\alpha}\in C([0,T];H^{s}(\mathbb{R}))$. A similar reasoning as above, using that $||\xi|^{\alpha/2}-|\xi|^{\alpha'/2}|\lesssim \langle \xi \rangle^{\beta/2}$ whenever $\alpha,\alpha'\in (0,\beta]$ establishes continuity when $\alpha\in(\alpha,\beta)$. The proof of \eqref{ContMap} is now completed. 


\section{fDBO on the Torus}

In this section we briefly indicate the modifications needed to prove Theorem \ref{periodiccase}. The periodic Sobolev spaces $H^{s}(\mathbb{T})$ are endowed with the norm 
\begin{align*}
 \left\|\phi\right\|_{H^s}^2=\sum_{k\in \mathbb{Z}} \langle k \rangle^{2s}|\widehat{\phi}(k)|^2.
\end{align*}
Let $s\in \mathbb{R}$ and $0<t\leq T\leq 1$ fixed. We consider the spaces
    $$\widetilde{Y}^s_T=\left\{u\in C([0,T];H^s(\mathbb{T})): \, \left\|u\right\|_{\tilde{Y}_T^s}<\infty \right\},$$
where
\begin{equation} \label{adaptyt}
    \left\|u\right\|_{\tilde{Y}_T^s}:= \sup_{t\in(0,T]} \left(\left\|u(t)\right\|_{H^s(\mathbb{T})}+t^{|s|/\beta}\left\|u(t)\right\|_{L^2(\mathbb{T})}\right).
\end{equation}
Note that when $s\geq 0$, $\tilde{Y}_T^s=C([0,T];H^s(\mathbb{T}))$ and $\left\|u\right\|_{\tilde{Y}_T^s}\sim \left\|u\right\|_{L^{\infty}_TH^s_x}$, since $0<T\leq 1$.

Comparing with the real line case, here the semigroup $\left\{S(t)\right\}_{t\geq 0}$ on $H^s(\mathbb{T})$ is contractive which clearly follows from the fact that $|k|^{\alpha}-|k|^{\beta} \leq 0$ for all integer $k$ and so $e^{(|k|^{\alpha}-|k|^{\beta})t} \leq 1$ for all $t\geq 0$.

On the other hand, since the proof of Proposition \ref{aprop1} and \ref{aprop2} depends on some change of variables, we must proceed with a bit more care. 

\begin{prop}
The results in Proposition \ref{aprop1} and those in Proposition \ref{aprop1.1} setting $0<t \leq 1$ are still valid in the periodic case. 
\end{prop}
\begin{proof}
Let $m\geq 0$ and $s\in \mathbb{R}$. Since $|k|^{\alpha}-|k|^{\beta}\leq |k|^{\beta}/2$ when $|k|\geq 2^{\frac{1}{\beta-\alpha}}$, and $|k|^{\alpha}-|k|^{\beta} \leq 0$ for all integer $k$, we find
\begin{equation}\label{peridoesti}
\begin{aligned}
  \left\||k|^m\langle k \rangle^{s}e^{(|k|^{\alpha}-|k|^{\beta)t}}\right\|_{l^{2}(\mathbb{Z})}^2 &\lesssim \sum_{0\leq k\leq  2^{\frac{1}{\beta-\alpha}}}|k|^{2m}\langle k \rangle^{2s}+\sum_{  k > 2^{\frac{1}{\beta-\alpha}} } |k|^{2m}\langle k \rangle^{2s} e^{-|k|^{\beta}t} \\
        &\lesssim \sum_{0\leq k\leq  2^{\frac{1}{\beta-\alpha}} }|k|^{2m}\langle k \rangle^{2s}+\sum_{  k >  2^{\frac{1}{\beta-\alpha}}  } \int_{k-1}^k|\xi+1|^{2m}\langle \xi \rangle^{2s} e^{-|\xi|^{\beta}t}\, d\xi \\
        & \lesssim 1+\left\|\langle \xi \rangle^{s}e^{-\frac{|\xi|^{\beta}}{2}t}\right\|_{L^2(\mathbb{R})}^2+\left\||\xi|^m\langle \xi \rangle^{s}e^{-\frac{|\xi|^{\beta}}{2}t}\right\|_{L^2(\mathbb{R})}^2.
\end{aligned}
\end{equation}
Therefore, inequality \eqref{peridoesti} allow us to argue exactly as in the proof of Propositions \ref{aprop1} and \ref{aprop1.1} to derive analogous time decay estimates. Furthermore one can see that the factor $\psi_{\alpha,\beta}(t)$ is not needed to bound the exponential term $e^{-\frac{|\xi|^{\beta}}{2}t}$ in each of these estimates. Finally, since $0<t\leq 1$, the constant term on the right-hand side of \eqref{peridoesti} can be bounded by $t^{-r}$, for any $r>0$. 
\end{proof}

Gathering the above results we deduce that Proposition \ref{aprop2} is valid in the periodic setting. Thus, for the range $\beta>3/2$ with growth order $0<\alpha<\beta$, we can repeat the same arguments in the proof of  Theorem \ref{globalw} changing $H^s(\mathbb{R})$ and $Y^s_T$, respectively by $H^s(\mathbb{T})$ and $\tilde{Y}^s_T$. This concludes the GWP part of Theorem \ref{periodiccase}.

Next we show Theorem \ref{illpos} (i) for the periodic case. Here, we define the function $u_0$ via its Fourier series by
\begin{equation}
    \widehat{u_0}(k)= \left\{\begin{aligned}
     &N^{-s}, \hspace{0.5cm} &&\text{ if } k=N \text{ or } k=1-N, \\
     &0,  &&\text{otherwise},  \\
    \end{aligned}\right.
\end{equation}
for $N\gg 1$. Noting that $\sigma(1,N)=\sigma(1,1-N)$ (with $\sigma$ defined by \eqref{ipeq}), it is deduced that
\begin{equation}
\begin{aligned}
        \widehat{u_2}(1,t)
        &=N^{-2s}e^{-it}\frac{e^{\sigma(1,N)t}-1}{\sigma(1,N)}.
        \end{aligned}
\end{equation}
Therefore since $|\sigma(1,N)|\sim N^{\beta}$ for $N$ large, we can follow the ideas behind \eqref{IIPR8} with $t_N=N^{-\beta-\epsilon}$ to obtain
\begin{equation}\label{IT1}
    \begin{aligned}
    \left\|u_2(t_N)\right\|_{H^s}\gtrsim  \left|\widehat{u_2}(1,t_N)\right|\gtrsim N^{-2s-\beta-\epsilon}.
    \end{aligned}
\end{equation}
Thus, \eqref{IT1} contradicts \eqref{IIP1} ($k=2$) given that $\left\|u_0\right\|_{H^{s}}\sim 1$ and $s<-\beta/2$ with  $0<\epsilon \ll 1$.
\\ \\
Finally, we discuss how to extend the conclusions of Proposition \ref{contInsta} to the periodic setting. In this context one can follow the same ideas dealing with the real line case, employing the periodic Kato-Ponce type inequality deduced in \cite{PeriKP} and replacing Lebesgue dominated convergence theorem by Weierstrass M-test. It is worth to emphasize that since $|k|^{\alpha}-|k|^{\beta}\leq 0$ for all integer $k\in \mathbb{Z}$ some of the estimates in the proof of Proposition \ref{contInsta} are simplified in the present case. For instance, the corresponding equation to \eqref{wleq0} satisfies
\begin{equation}
    \int_{\mathbb{T}}( D_x^{\alpha}-D_x^{\beta})J^{s_0}u^{\alpha}(t)J^{s_0}u^{\alpha}(t) \, dx= \sum_k (|k|^{\alpha}-|k|^{\beta})|\widehat{J^{s_0}u^{\alpha}}(k,t)| \leq 0
\end{equation}
and in a similar way, the estimate \eqref{EstimateI} adapted to $H^s(\mathbb{T})$ assures that $\mathcal{I}_{\alpha}\leq 0$.
This encloses all the conclusions stated in Theorem \ref{periodiccase}. 

\subsection*{Acknowledgements}

The author O.R. acknowledges support from CNPq-Brazil. R.P. was supported by the Universidad Nacional de Colombia, Bogot\'a.


\bibliographystyle{amsplain}

\end{document}